\documentclass[a4paper,11pt,reqno]{amsart}
\usepackage[utf8]{inputenc}
\usepackage[T1]{fontenc}
\usepackage{lmodern}
\usepackage[english]{babel}
\usepackage{amsmath,a4wide,wasysym}
\usepackage{xfrac}
\usepackage{esint}
\usepackage{stmaryrd,mathrsfs,bm,amsthm,mathtools,yfonts,amssymb,color,braket,booktabs,graphicx,graphics,amsfonts}
\usepackage{latexsym,microtype,indentfirst,hyperref}
\usepackage{xcolor}
\usepackage{courier}
\usepackage[colorinlistoftodos]{todonotes}

\begingroup
\newtheorem{theorem}{Theorem}[section]
\newtheorem{lemma}[theorem]{Lemma}
\newtheorem{proposition}[theorem]{Proposition}
\newtheorem{corollary}[theorem]{Corollary}

\endgroup

\theoremstyle{definition}
\newtheorem{definition}[theorem]{Definition}
\newtheorem{remark}[theorem]{Remark}
\newtheorem{example}[theorem]{Example}
\newtheorem{ipotesi}[theorem]{Assumption}

\numberwithin{equation}{section}

\DeclareMathAlphabet{\mathpzc}{OT1}{pzc}{m}{it}

\newcommand{\N}{\mathbb{N}} 
\newcommand{\R}{\mathbb{R}} 

\newcommand{\bC}{\mathbf{C}} 

\newcommand{\p}{\mathbf{p}} 

\newcommand{\A}{\mathcal{A}} 
\newcommand{\G}{\mathcal{G}} 
\newcommand{\bfeta}{\boldsymbol{\eta}} 
\newcommand{\Dir}{\mathrm{Dir}} 
\newcommand{\gr}{\mathrm{Gr}}
\newcommand{\zetab}{{\bm{\zeta}}}
\newcommand{\varrhob}{{\bm{\varrho}}}
\newcommand{\xib}{{\bm{\xi}}}
\newcommand{\rhob}{{\bm{\rho}}}

\newcommand{\mass}{\mathbf{M}} 
\newcommand\res{\mathop{\hbox{\vrule height 7pt width .3pt depth 0pt\vrule height .3pt width 5pt depth 0pt}}\nolimits}

\newcommand{\reg}{\mathrm{Reg}} 
\newcommand{\sing}{\mathrm{Sing}} 

\newcommand{\bphi}{\boldsymbol{\varphi}}

\newcommand{\modp}{{\rm mod}(p)} 

\newcommand{\Ha}{\mathcal{H}} 

\newcommand{\eps}{\varepsilon} 
\newcommand{\spt}{\mathrm{spt}} 
\newcommand{\dist}{\mathrm{dist}} 
\newcommand{\Lip}{\mathrm{Lip}} 
\renewcommand{\epsilon}{\varepsilon}

\def\XXint#1#2#3{{\setbox0=\hbox{$#1{#2#3}{\int}$ }
\vcenter{\hbox{$#2#3$ }}\kern-.6\wd0}}

\newcommand{\ssubset}{\subset\joinrel\subset}

\newcommand{\Iqs}{{\mathcal{A}}_Q(\R^{n})}

\def\a#1{\left\llbracket{#1}\right\rrbracket}
\newcommand{\abs}[1]{\lvert#1\rvert} 
\newcommand{\norm}[1]{\left\lVert#1\right\rVert} 
\newcommand{\etab}{\boldsymbol{\eta}}

\newcommand{\Iqspec}{{\mathscr{A}_Q (\R^n)}}
\newcommand{\Iqsn}{\overset{\circ}{\mathcal{A}_Q}(\R^{n})}

\newcommand{\IQSn}{\overset{\circ}{\mathscr{A}_Q}(\R^{n})}

\def\a#1{\left\llbracket{#1}\right\rrbracket}
\newcommand{\iso}{\boldsymbol{\iota}}

\newcommand\cM{{\mathcal{M}}}

\newcommand\Phii{{\mathbf{\Phi}}}

\newcommand\phii{{\mathbf{\varphi}}}

\newcommand\cH{{\mathcal{H}}}
\newcommand\bU{{\mathbf{U}}}

\newcommand\bT{{\mathbf{T}}}
\newcommand\cG{{\mathcal{G}}}

\newcommand\bef{{\mathbf{f}}}
\newcommand\beg{{\mathbf{g}}}

\begin{document}

\title[Area minimizing currents mod $2Q$: linear regularity theory]{Area minimizing currents mod $2Q$: \\ linear regularity theory}

\author{Camillo De Lellis}
\address{School of Mathematics, Institute for Advanced Study, 1 Einstein Dr., Princeton NJ 05840, USA,\\
and Universit\"at Z\"urich}
\email{camillo.delellis@math.ias.edu}

\author{Jonas Hirsch}
\address{Mathematisches Institut, Universit\"at Leipzig, Augustusplatz 10, D-04109 Leipzig, Germany}
\email{hirsch@math.uni-leipzig.de}

\author{Andrea Marchese}
\address{Dipartimento di Matematica, Universit\`a degli Studi di Trento, Via Sommarive 14, I-38123 Povo (TN), Italy}
\email{andrea.marchese@unitn.it}

\author{Salvatore Stuvard}
\address{Department of Mathematics, The University of Texas at Austin, 2515 Speedway, Stop C1200, Austin TX 78712-1202, USA}
\email{stuvard@math.utexas.edu}

\vspace{0.5cm}

\begin{abstract}
We establish a theory of $Q$-valued functions minimizing a suitable generalization of the Dirichlet integral. In a second paper the theory will be used to approximate efficiently area minimizing currents $\modp$ when $p=2Q$, and to establish a first general partial regularity theorem for every $p$ in any dimension and codimension.\\


\textsc{AMS Math Subject Classification (2010):} 49Q15, 49Q05, 49N60, 35B65, 35J47.
\end{abstract}

\maketitle


\medskip

\section{Introduction}

The aim of this work and its companion paper \cite{DLHMS} is to give a proof of the following partial regularity theorem (for the definition of area minimizing currents $\modp$ and the relevant terminology and notation we refer to \cite{DLHMS}):

\begin{theorem}\label{t:main_modp}
Assume $p\in \mathbb N\setminus \{0,1\}$ and $a_0>0$, $\Sigma\subset \mathbb R^{m+n}$ is a complete $C^{3, a_0}$ submanifold without boundary of dimension $m+\bar n$, $\Omega\subset \mathbb R^{m+n}$ is open and $T$ is an $m$-dimensional integer rectifiable current supported in $\Sigma$ which is area minimizing $\modp$ in $\Omega\cap \Sigma$. Then, the interior singular set $\sing (T)$ of $T$ has Hausdorff dimension at most $m-1$. If $p$ in addition is odd, then the singular set is countably $(m-1)$-rectifiable.
\end{theorem}

The above result provides an affirmative answer in full generality to a question of B. White; see \cite[Problem 4.20]{open_GMT}. Prior to our work, some of the conclusions above were only known in some special cases. More precisely, in general codimension $\bar n >1$:
\begin{itemize}
\item[(a)] For $m =1$ it is elementary that $\sing (T)$ is discrete (and empty when $p=2$);
\item[(b)] In general, Allard's interior regularity theory for stationary varifolds, cf. \cite{Allard72}, implies that $\sing (T)$ is a closed meager set in $(\spt^p (T)\cap \Omega)\setminus \spt^p (\partial T)$;
\item[(c)] For $p=2$ $\sing (T)$ has Hausdorff dimension at most $m-2$ by Federer's classical work \cite{Federer70}; moreover the same reference shows that such set is in fact discrete when $m=2$; for $m>2$ its $(m-2)$-rectifiability was first proved in \cite{Simon95}, and the recent work \cite{NV} implies in addition that it has locally finite $\mathcal{H}^{m-2}$ measure.
\end{itemize}
In the case of codimension $\bar n =1$ it was additionally known that:
\begin{itemize}
\item[(d)] When $p=2$, the singular set has $(m-2)$-dimensional Hausdorff measure zero even in the case of minimizers of general uniformly elliptic integrands, see \cite{ASS}; for the area functional, using \cite{NV}, one can conclude additionally that it is $(m-3)$-rectifiable and has locally finite $\mathcal{H}^{m-3}$ measure; 
\item[(e)] When $p=3$ and $m=2$, \cite{Taylor} gives a complete description of the singularities, which consist of $C^{1, \alpha}$ arcs where three regular sheets meet at equal angles;
\item[(f)] When $p$ is odd, \cite{White86} shows that the singular set has vanishing $\mathcal{H}^m$-Hausdorff measure for minimizers of a uniformly elliptic integrand, and that it has Hausdorff dimension at most $m-1$ for minimizers of the area functional;
\item[(g)] When $p=4$, \cite{White79} shows that minimizers of uniformly elliptic integrands are represented by {\em immersed manifolds} outside of a closed set of zero $\mathcal{H}^{m-2}$ measure. 
\end{itemize}
Our proof of Theorem \ref{t:main_modp} follows the blueprint of Almgren's partial regularity theory for area minimizing currents as worked out in the papers \cite{DLS_Qvfr,DLS_Currents,DLS_Lp,DLS_Center,DLS_Blowup}. First of all, thanks to the general stratification theorem of the singular set, for every $\alpha >0$ we know that at $\mathcal{H}^{m-1+\alpha}$-a.e. $x\in \spt^p (T)\setminus \spt^p (\partial T)$ there is at least one tangent cone which is flat, namely an integer multiple of an $m$-dimensional plane. If we call such points ``flat'', the main dimension estimate in Theorem \ref{t:main_modp} is achieved by showing that, for every $\alpha>0$, $\mathcal{H}^{m-1+\alpha}$-a.e. flat point $x$ is in fact regular. Every flat point $x$ where the density of $T$ is $1$ is indeed regular by Allard's celebrated theorem. The problem arises when the multiplicity is higher than $1$, because there are examples of singular flat points. For area minimizing {\em integral} currents such examples exist only in codimension $\bar n \geq 2$, whereas for area minimizing currents $\modp$ such examples can be found also in codimension $\bar n =1$ if $p$ is even and larger than $2$, see for instance Example \ref{e:esempio} below. 

An essential step in Almgren's theory is the approximation of the area minimizing currents, in regions where they are sufficienly close to an integer multiple of a plane, with multivalued functions which almost minimize an appropriate generalization of the Dirichlet energy. We will call ``\emph{linear theory}'' the corresponding existence and regularity theory for those objects. In the case of integral currents a typical example where multivalued functions are needed is in the approximation of the current $\a{\Lambda}$ induced by the holomorphic curve 
\[
\Lambda = \{(z,w)\in \mathbb C^2 : z^2 = w^3\}
\] 
in a neighborhood of the origin (which is indeed a singular flat point of multiplicity $2$). One way of understanding multiple-valued functions which take a fixed number $Q$ of values is to model them as maps into the space of atomic measures with positive integral coefficients and total mass $Q$. For instance, slicing the the current $\Lambda$ with (real) two-dimensional planes orthogonal to $\{(z,0): z\in \mathbb C\}$, for each $z\in \mathbb C\setminus \{0\}$ we find an integral $0$-dimensional current which is the sum of two positive atoms:
\[
\sum_{w^3 =z^2} \a{(z,w)}\, .
\]
Such maps can be efficiently used to approximate area-minimizing currents $T$ $\modp$ in a neighborhood of a flat point $x$ when
\begin{itemize}
\item either $p$ is odd;
\item or $p$ is even and the density $Q$ of $T$ is strictly smaller than $\frac{p}{2}$.
\end{itemize}
When studying area-minimizing currents $\modp$ for an even modulus $p=2Q$ in a neighborhood of a flat point of density $Q$, the ``classical'' multivalued functions are not anymore the appropriate maps, as it is witnessed by the following example, taken from \cite{White79}. 

\begin{example}\label{e:esempio}
Consider an open subset $\Omega \subset \mathbb R^2$ and two smooth functions $f,g: \Omega \to \mathbb R$ which solve the minimal surfaces equation in $\Omega$. Assume in addition that the set $\{f=g\}$ contains a curve $\gamma$  which divides $\Omega$ into two regions $\Omega_+$ and $\Omega_-$. Two explicit $f$ and $g$ are easy to find. The reader could take $\Omega$ to be a suitable ball $B$ centered at the origin, $f\equiv 0$, and let $g$ be the function which describes Enneper's minimal surface in a neighborhood of $0$. The set $\{f=g\}$ is then given by $\{(x,y) : x=\pm y\}\cap B$ and $\gamma$ can be taken to be the segment $\{x=y\}\cap B$ while $\Omega_+$ and $\Omega_-$ would then be $B \cap \{x>y\}$ and $B \cap \{x<y\}$, respectively. 

We then define the following integral current $T$. Its support is the union of the graphs of $f$ and $g$. However, while the portions of such graphs lying over $\Omega_+$ will be taken with the standard orientation induced by $\Omega$, the portions lying over $\Omega_-$ will be taken with the opposite orientation. In $\Omega\times \mathbb R$, the boundary of $T$ is $4 \a{\gamma}$. Moreover, by the structure theorem \cite{White79}, the current is area minimizing $\mod (4)$, because the graphs of $f$ and $g$ are both area minimizing currents $\mod (2)$ (this could be proved using, for instance, the maximum principle). 
\end{example}

The origin is a flat point of multiplicity $Q=2$ for the current $T$ above. By a simple rescaling procedure a good approximation of $T$ in a neighborhood of the origin is given by the graphs of the second order Taylor polynomials of $f$ and $g$ at the origin. These are harmonic polynomials. For the specific case described above where $f=0$ and the graph of $g$ is Enneper's surface, such functions are $f_0 (x,y) =0$ and $g_0 (x,y) = 3 (x^2-y^2)$. This gives an obvious set-theoretic approximation of the support of the current $T$. In the approach that we outline in the rest of the paper, we will give to this set a structure of ``special $2$-valued function'' $h$, where we consider the value $h(x,y)$ to be the sum of the two positive atoms $\a{f_0 (x,y)} + \a{g_0 (x,y)}$ on $\Omega_+ = B \cap \{x>y\}$ and the sum of two negative atoms $-\a{f_0 (x,y)} - \a{g_0 (x,y)}$ on $\Omega_- = B \cap  \{x<y\}$. Such a choice is natural in view of the fact that the slices of the current $T$ with lines orthogonal to the plane $\{(x,y,0): x,y\in  \mathbb R\}$ are given by $\a{f (x,y)} + \a{g (x,y)}$ for $(x,y)\in \Omega_+$ and $- \a{f (x,y)} - \a{g (x,y)}$ for $(x,y)\in \Omega_-$.

Motivated by the above example, roughly speaking ``special $2$-valued functions'' will be maps from $\Omega$ into the space of atomic measures with mass $2$ satisfying the following requirements (cf. Definition \ref{def:Iqstronzo} and Definition \ref{d:pieces}):
\begin{itemize}
\item The value of the map at any point in $\Omega$ is always either the sum of two positive atoms or the sum of two negative atoms;
\item The domain $\Omega$ is subdivided by each map into three regions, the ``positive region'' where the values are two positive {\em disinct atoms}, the ``negative region'' where the values are two {\em distinct} negative atoms and the ``interface'', or "collapsed region", where the values are atoms counted with multiplicity $2$: whether with a plus or minus sign, this will be of no relevance, because we will identify $-2 \a{z}$ and $2\a{z}$ (which are equivalent $0$-dimensional currents $\mod (4)$).
\end{itemize}
Roughly speaking, if the special $2$-valued map is continuous, then the collapsed region disconnects the ``positive'' and the ``negative'' ones. 

A natural Dirichlet energy, which comes out of Taylor expanding the area functional on the original current, is the sum of the Dirichlet energies of the various sheets: with such definition, the special $2$-valued function $h$ considered above is a minimizer of the Dirichlet energy, namely any competitor which coincides with it outside a compact set $K\subset \subset B$ has at least the same energy. This could be proved in an elementary way in our specific example, but it is also a general fact.

The reader might wonder why we introduce such complicated objects, rather than simply considering the union of 
the two graphs of $f_0$ and $g_0$ as a classical $2$-valued function (namely, always taking positive atomic measures as values) as in \cite{DLS_Qvfr}). The point is that with the latter choice, the resulting $2$-valued function would not be a minimizer of the Dirichlet energy. A better competitor could be easily constructed by considering the following functions $\bar f$ and $\bar g$: both are harmonic in $B_1 (0)$ and their values on $\partial B_1 (0)$ are, respectively:
\begin{equation}
\bar f (x,y) = 
\left\{\begin{array}{ll}
3 (x^2-y^2) & \mbox{if $|x| \geq |y|$}\\
0 & \mbox{if $|x|\leq |y|$}\\
\end{array}\right.
\end{equation}
\begin{equation}
\bar g (x,y) = 
\left\{\begin{array}{ll}
0 & \mbox{if $|x| \geq |y|$}\\
3 (x^2-y^2) & \mbox{if $|x|\leq |y|$}\\
\end{array}\right.
\end{equation}

The example above also shows that the regularity theory for Dirichlet-minimizing special $Q$-valued functions must necessarily allow for a larger set of singularities than its classical counterpart: indeed, for the special $2$-valued map $h$ constructed above any reasonable definition of the singular set ${\rm Sing}(h)$ must be such that $\{ x = \pm y \} \cap B \subset {\rm Sing}(h)$, thus implying that the standard result $\dim_{\Ha}({\rm Sing}(u)) \leq m-2$ valid for a classical $Q$-valued map $u$ defined on an $m$-dimensional domain and minimizing the Dirichlet energy (or even natural perturbations of the Dirichlet energy, see e.g. \cite{St_Jac}) cannot hold true in our context. \\

\smallskip

Before proceeding with our analysis, let us remark that, in the paper \cite{Almgren_modp}, F. Almgren seems to initiate the investigation of a class of objects which are conceptually analogous to our special multiple valued functions. More precisely, Almgren's ``multi-functions $\modp$'' are defined as mappings taking values in the space of $0$-dimensional integral polyhedral chains $\modp$. The theory outlined in \cite{Almgren_modp} may have some points in common with the content of Sections \ref{sec:def} and \ref{sec:ispec_curr} of the present work, as well as Section $10$ of \cite{DLHMS}. The Dirichlet energy and the corresponding regularity theory, on the other hand, are not mentioned in \cite{Almgren_modp}, which rather seems to focus on describing the geometric properties of a class of  \emph{piecewise affine} multi-functions, which have the property to induce, via push-forward, dimension-preserving homomorphisms of the space of polyhedral chains. Since Almgren did not pursue this line of research anymore in later works, we don't know whether his ultimate goal was to seek a regularity theory for minimizing currents $\modp$ along the lines of his Big Regularity Paper \cite{Almgren_brp}.

\subsection{Plan of the paper} The first part of the paper aims at establishing the optimal partial regularity result for special $Q$-valued functions minimizing the Dirichlet energy. After providing the precise definition of the space $\mathscr{A}_Q(\R^n)$ of special $Q$-points in $\R^n$ and introducing the corresponding Sobolev spaces of $\mathscr{A}_Q(\R^n)$-valued maps, we show that any $\Dir$-minimizing special $Q$-valued function $u$ is H\"older continuous with respect to the natural metric space structure of $\mathscr{A}_Q(\R^n)$, and then that the - suitably defined - set ${\rm Sing}(u)$ of singular points of $u$ is a closed subset of the $m$-dimensional domain of $u$ having Hausdorff dimension $\dim_{\Ha}({\rm Sing}(u)) \leq m-1$. We will then conclude the paper with some results concerning the geometry of (the currents associated to) the graphs of special multiple-valued functions, which will be crucial for the analysis to be carried out in \cite{DLHMS}.\\

\medskip

{\bf Acknowledgments:} C.D.L. acknowledges the support of the NSF grants DMS-1946175 and DMS-1854147. A.M. was partially supported by INdAM GNAMPA research projects. The work of S.S. was supported by the NSF grants DMS-1565354, DMS-RTG-1840314 and DMS-FRG-1854344.

\section{Definition of $\Iqspec$ and metric properties} \label{sec:def}

For the classical $Q$-valued maps in $\R^n$, denoted $\Iqs$, we follow the terminology, notation and definitions
of \cite{DLS_Qvfr}. We first introduce the disjoint union  $\Iqs \sqcup\Iqs$, which we identify with $\Iqs \times \{-1,1\}$. Hence, an element
in $\Iqs \sqcup\Iqs$ will be denoted by $(S, \epsilon)$, where $S$ is an element of the space $\Iqs$ of atomic measures with
positive integer coefficients and mass $Q$ (namely $S = \sum_{i=1}^Q \a{P_i}$ for $P_i \in \mathbb R^n$) and $\epsilon$ equals either $1$ or $-1$.

Moreover, it is convenient to introduce the following notation. Recall that $\G(\cdot,\cdot)$ denotes the distance function in $\A_Q(\R^n)$.

\begin{definition}
If $S= \sum_i \a{S_i}\in \Iqs$ and $v\in \mathbb R^n$, then $|S|^2 := \G (S, Q\a{0})^2$ and
\begin{align*}
S \oplus v & := \sum_i \a{S_i+v}\\
S \ominus v & := S \oplus (-v) = \sum_i \a{S_i-v}\, .
\end{align*}
\end{definition}

Note that, using $\etab(S):=\frac{1}{Q} \sum_i S_i$, we get 
\begin{align}\label{eq:norms and averages1}
|S|^2 &= |S\ominus\etab (S)|^2 + Q|\etab (S)|^2\\
\G (A,B)^2 &= \G (A\ominus \etab (A), B \ominus \etab (B))^2 +Q |\etab (A) - \etab (B)|^2 \label{eq:norms and averages2}
\end{align}

\begin{definition}\label{def:Iqstronzo}
We denote by $\Iqspec$ the quotient space 
\[ \Iqspec:= \Iqs \sqcup\Iqs / \sim \]
where $\sim$ is the equivalence relation given by
\begin{align}
(S,-1) \sim (T,1) &\iff\; \exists p \in \R^n \text{ with } S=Q\a{p}=T\, ,\\
(S,1) \sim (T,1)  &\iff\; S = T\, ,\\
(S,-1) \sim (T,-1) &\iff\; S=T\, .
\end{align}
We endow $\Iqspec$ with the metric
\begin{equation}\label{eq:stronzo-metric}
\G_s((S,\alpha), (T,\beta))^2= \begin{cases} \G(S,T)^2 &\text{ if } \alpha = \beta \\ \\
|S\ominus \etab (S)|^2 + |T\ominus\etab (T)|^2 + Q |\etab (S) - \etab (T)|^2 
&\text{ if } \alpha \neq \beta.
\end{cases}
\end{equation}
\end{definition}

\begin{remark} We can consider $\G_s$ as a pseudometric in $\Iqs \sqcup\Iqs$: $\Iqspec$ results then from quotienting the corresponding pseudometric space to a metric space. It is hence straightforward to check that the quotient space topology coincides with the metric topology generated by $\mathcal{G}_s$. \\
Furthermore, for each $\alpha \in \{-1,1\}$ the injection $i_\alpha : \Iqs \ni S \mapsto (S, \alpha)\in \Iqspec$ is an isometry. 
\end{remark}

Given the identification of $(Q\a{p}, 1)$ with $(Q\a{p}, -1)$, in the sequel we will often use the simplified notation $Q\a{p}$ to denote both points in $\Iqspec$.

Since working with the above definition of $\Iqspec$ is sometimes inconvenient, we will next provide a useful characterization. 
We start by introducing the convention that, if $(X,d)$ and $(Y, \delta)$ are two metric spaces, then, unless otherwise specified, we endow the product space $X \times Y$ with the product
metric 
\[
d\times \delta ((x,y), (v,w)) := \sqrt{d (x,v)^2 + \delta (y,w)^2}\, .
\]

\begin{definition}
We denote by
\begin{itemize}
\item $\Iqsn$ the space $\{ T \in \Iqs \colon \etab(T)= 0\} \subset \Iqs$ endowed with the metric $\G$;
\item $\IQSn$ the space $\{ (T,S) \in \Iqsn \times \Iqsn \colon \min\{\abs{T}, \abs{S}\} = 0 \}$ endowed with the metric $\G\times \G$.
\end{itemize}
\end{definition}

\begin{remark}
Observe that 
\[
\IQSn = \left(\Iqsn\times \{Q \a{0}\}\right) \cup \left( \{Q\a{0}\} \times \Iqsn\right)  \subset \Iqsn\times \Iqsn\, .
\]
\end{remark}

\begin{proposition}\label{prop.isometry}
Consider the metric spaces $(\IQSn, \G\times \G)$ and $(\mathbb R^n, d)$ where 
\[
d(x,y) = \sqrt{Q} |x-y|\, .
\]
Endow the product $\IQSn\times \mathbb R^n$ with the corresponding product metric $(\G\times \G)\times d$. Then the map $\iso: \Iqspec \to \IQSn\times \mathbb R^n$ given by
\[
\iso (T, \epsilon) :=
\begin{cases} (T\ominus \etab(T) , Q \a{0}, \etab(T)) &\text{ if } \epsilon =1\,,\\ \\ 
(Q \a{0}, T \ominus \etab(T) , \etab(T)) & \text{ if } \epsilon =-1\end{cases}
\]
is an isometry with inverse
\[
\iso^{-1} (A,B, p) = \begin{cases} (A \oplus p , 1) & \text{ if } \abs{B}=0 \\ \\
(B \oplus p, -1) & \text{ otherwise}. \end{cases}
\]
\end{proposition}

In view of the previous proposition the metric $\G\times \G$ will be denoted by $\G_s$ when restricted to $\IQSn$. 

\begin{proof}
It is clear that the maps $\iso$ and $\iso^{-1}$ are well defined, and it is also obvious that $\iso \circ \iso^{-1}$ and $\iso^{-1}\circ \iso$ are the identity maps of the appropriate spaces.

Next, if we endow $\Iqsn\times \R^n$ with the product metric $\G\times d$, by \eqref{eq:norms and averages2} it is obvious that the map $\Iqs \ni A \mapsto (A \ominus \etab(A), \etab (A))\in \Iqsn\times \R^n$ is an isometry with inverse $(A, v) \mapsto A \oplus v$. In particular this shows that, for any fixed $\varepsilon \in \{-1,1\}$, the following holds
\[
((\G\times \G) \times d) (\iso (T,\epsilon), \iso (S, \epsilon)) = \G_s ((T, \epsilon), (S, \epsilon))\, .
\]
On the other hand the identity $((\G\times \G) \times d) (\iso (T,1), \iso (S, -1)) = \G_s ((T, 1), (S, -1))$ is obvious from the definition of $\G_s$. 
\end{proof}

For further use, it is very convenient to introduce the following notations:

\begin{definition}\label{d:pieces}
Let $u: E \to \Iqspec$ be a Borel map, and consider the map $(v,w, z) = \iso \circ u$. We then define:
\begin{align}
\etab \circ u &:= z\label{e:mean_spec}\\
u^+ &:= v \oplus z\label{e:positive_spec}\\
u^- &:= w \oplus z\label{e:negative_spec}\\
E_+ &:= \{|v| > 0\}\label{e:positive_domain}\\
E_- &:= \{|w| > 0\}\label{e:negative_domain}\\
E _0 &:= \{|v|=|w| = 0\}\label{e:collapsed_domain}\, .
\end{align}
Note in particular that $E_+, E_-$ and $E_0$ are pairwise disjoint and their union is $E$: $E_+, E_-$ and $E_0$ will be called the {\em canonical decomposition} of $E$ induced by the map $u$. These sets are those loosely described as positive, negative and collapsed regions in the example discussed in the introduction.  

Similarly, consider a point $P = (R,S, z)\in \IQSn\times \mathbb R^n$ and a vector $z'\in \mathbb R^n$. We denote by $P\oplus z'$, resp. $P \ominus z'$, the
points $(R, S, z+z')$ and $(R,S, z-z')$. 
\end{definition}

The following is thus an obvious corollary of Proposition \ref{prop.isometry}. 

\begin{corollary}\label{c:separate_Lip}
Let $u: E \to \Iqspec$ be Lipschitz. Then $E_+, E_-\subset E$ are relatively open and $E_0\subset E$ is relatively closed. Moreover $\etab\circ u, u^+$ and $u^-$ are all Lipschitz and their Lipschitz constants are at most $\Lip (u)$. 
More generally, if $u$ is merely continuous, then $\etab\circ u, u^+$ and $u^-$ are also continuous and their moduli of continuity are at most that of $u$. 
\end{corollary}

Recall that any Lipschitz map $F: \R^{\bar{n}} \to \R^n$ induces a natural map $F:\mathcal{A}_Q (\R^{\bar{n}}) \to \mathcal{A}_Q (\R^n)$ via 
\[
F \left(\sum_i \a{T_i}\right) := \sum_i \a{F(T_i)}\, ,
\]
and hence a natural map $F:\mathscr{A}_Q(\R^{\bar{n}}) \to \mathscr{A}_Q(\R^n)$ by
\[ F((T, \alpha)):= (F(T), \alpha) = \left(\sum_{i} \a{F(T_i)}, \alpha\right) \qquad \mbox{if $T = \sum_{i} \a{T_i}$}\,. \]
In terms of the identification above we have 
\begin{align*}
	 \left(\iso\circ F\circ \iso^{-1}\right)\left( (R,S,z)\right) &= \begin{cases} \left( F(R\oplus z) \ominus \etab (F(R\oplus z)), Q\llbracket 0 \rrbracket,   \etab (F(R\oplus z))\right)& \text{ if } S =Q \llbracket 0 \rrbracket \\ \left( Q \llbracket 0 \rrbracket, F(S\oplus z) \ominus \etab (F(S\oplus z)), \etab (F(S\oplus z)) \right) &\text{ if } R =Q\llbracket 0 \rrbracket\,. \end{cases}
\end{align*}

\section{Sobolev spaces, differentiability and Dirichlet energy}

The embedding $\iso$ allows to provide a straightforward definition of the Sobolev spaces $W^{1,p} (\Omega, \Iqspec)$ using the theory developed in \cite{DLS_Qvfr}. 
Similarly, we shall define the Dirichlet energy and its density.

\begin{definition}\label{d:Sobolev}
Let $\Omega$ be an open subset of a $C^1$ manifold. 
We say that the function $u: \Omega \to \Iqspec$ belongs to the Sobolev space $W^{1,p} (\Omega, \Iqspec)$ if each of the maps $v,w,z$ given
by $\iso (u) = (v,w,z)$ belongs to the respective $W^{1,p}$ space. 

If $u\in W^{1,2}$ we then define $|Du|^2 := |Dv|^2 + |Dw|^2 + Q |Dz|^2$ and the corresponding Dirichlet energy
\[
\Dir (u, \Omega) := \int_\Omega |Du|^2 = \Dir (v, \Omega) + \Dir (w, \Omega) + Q \, \Dir (z, \Omega)\, .
\]
\end{definition}

Observe the validity of the identity (which holds as well for the ``classical'' $Q$-valued $W^{1,p}$ spaces)
\begin{equation}\label{e:split_energy}
\Dir (u, \Omega) = \Dir (u\ominus \etab \circ u, \Omega) + Q \, \Dir (\etab\circ u, \Omega)\, .
\end{equation}

Using the definition above, one concludes obviously the analogues of 
\begin{itemize}
\item The Lipschitz extension theorem, cf. \cite[Theorem 1.7]{DLS_Qvfr};
\item The trace theorem, cf. \cite[Proposition 2.10]{DLS_Qvfr};
\item The Sobolev embedding theorem, cf. \cite[Proposition 2.11]{DLS_Qvfr};
\item The Poincar\'e inequality, cf. \cite[Proposition 2.12]{DLS_Qvfr};
\item The Campanato-Morrey estimate of \cite[Proposition 2.14]{DLS_Qvfr}.
\end{itemize}
From now on we will use all the results above referring to the corresponding statements in \cite{DLS_Qvfr}. 

Next, it is useful to gain a local description of $|Du|$ in terms of the differentials of the maps $u^+, u^-$ and $\etab \circ u$. In particular this will allow
us to apply the calculus tools of \cite{DLS_Qvfr} making several computations straightforward.

\begin{proposition}\label{p:app_differentiability}
Assume $u\in W^{1,2} (\Omega, \Iqspec)$. The maps $u^+, u^-$ and $\etab\circ u$ are approximately differentiable at a.e. point $x\in \Omega$.
In particular, if we denote by $D u^+ = \sum_i \a{Du_i^+}$, $Du^- = \sum_i \a{Du_i^-}$ and $D (\etab\circ u)$ their approximate differentials (using the conventions
of \cite[Section 1.3 \& Section 2.2.1]{DLS_Qvfr}), then we have
\begin{equation}\label{e:id_differenziali}
|Du|^2 (x) = \begin{cases}
|Du^+|^2 (x) = \sum_i |Du_i^+|^2 \qquad &\mbox{for a.e. $x\in \Omega_+\cup \Omega_0$}\\
|Du^-|^2 (x) = \sum_i |Du_i^-|^2  \qquad &\mbox{for a.e. $x\in \Omega_-\cup \Omega_0$}\\
Q|D (\etab\circ u)|^2 (x) \qquad &\mbox{for a.e. $x\in \Omega_0$}\, .
\end{cases}
\end{equation}
\end{proposition}
\begin{proof}  Let $\iso (u) = (v,w, z)$.  From the very definition we know that $\etab \circ u =z$ belongs to $W^{1,2}(\Omega, \R^n)$. Next observe that, if $a\in W^{1,2} (\Omega, \Iqs)$ and $b\in W^{1,2} (\Omega, \mathbb R^n)$, then
$a\oplus b\in W^{1,2} (\Omega, \A_Q(\mathbb R^n))$, as one can easily check from \cite[Definition 0.5]{DLS_Qvfr}. Hence, $u^+, u^-$ 
belong to $W^{1,2}(\Omega, \Iqs)$. Thus, the approximate differentiability a.e. of $\etab\circ u$, $u^+$ and $u^-$ follows from \cite[Corollary 2.7]{DLS_Qvfr}. 

The approximate differentiability of $v,w$ and the fact that they are identically $Q\a{0}$ on $\Omega_0$ implies easily that
indeed $|Dv|=|Dw| =0$ a.e. on $\Omega_0$. This shows, therefore, the third case of \eqref{e:id_differenziali}. We now come to the other two cases and, by symmetry, we focus on the first one. Clearly, on $\Omega_+\cup \Omega_0$ we have $|Dw|=0$ and thus by definition
\[
|Du|^2 = |Dv|^2 + Q |D(\etab\circ u)|^2\, .
\]
On the other hand, on $\Omega_+\cup \Omega_0$ we also have that $\etab\circ u = \etab\circ u^+$ and that
\[
v = \sum_i \a{u_i^+ - \etab\circ u^+} = u^+ \ominus \etab \circ u^+\, .
\]
Now, at every point of approximate differentiability $x$ we readily check from \cite[Definition 1.9 \& Definition 2.6]{DLS_Qvfr} that
$D (\etab\circ u^+) (x) = \frac{1}{Q} \sum_i Du_i^+ (x)$ and that $Dv_i (x) = Du_i^+ (x) - D (\etab \circ u^+) (x)$. 
Recalling \cite[Proposition 2.17]{DLS_Qvfr} we have thus
\[
|Du^+|^2 (x) = \sum_i |Du_i^+ (x)|^2 = \sum_i |Dv_i|^2 + Q |D (\etab \circ u^+) (x)|^2 = |Dv|^2 (x) + Q |D (\etab\circ u) (x)|^2\, .
\]
The latter identity completes the proof. 
\end{proof}

\section{Currents $\mod (2Q)$ and $\Iqspec$-valued maps} \label{sec:ispec_curr}

In this section we link the notion of special $Q$-valued maps to that of currents modulo $2Q$. This will not only be very useful in the proof of Theorem \ref{t:main_modp} given in \cite{DLHMS}, but it also highlights the intuition behind the definition of $\Iqspec$ as described in the introduction. Consider a $k$-dimensional
rectifiable set $E\subset \mathbb R^m$ with finite $\mathcal{H}^k$ measure and a proper Lipschitz map $u: E\to \Iqspec$ (i.e. $\etab \circ u$, $u^+$ and $u^-$ are proper, see \cite[Definition 1.2]{DLS_Currents} for the definition of proper $\Iqs$-valued maps). We can use Definition \ref{d:pieces}, Corollary
\ref{c:separate_Lip} and the theory presented in \cite{DLS_Currents} to define a suitable notion of ``graph'' of $u$ and correspondingly associate a rectifiable current to it.

\begin{definition}\label{d:graph}
Let $E\subset \mathbb R^m$ be countably $k$-rectifiable with finite $\mathcal{H}^k$ measure and let $u: E\to \Iqspec$ be Lipschitz and proper.
Using the terminology of \cite{DLS_Currents} we denote by
\begin{itemize}
\item[(i)] $\gr (u)$ the set 
\[
\gr (u) := (\gr (u^+)\cap (E_+\times \R^n))  \cup (\gr (u^-) \cap (E_-\times \R^n)) \cup (\gr (\etab\circ u)\cap (E_0\times \R^n))\, ;
\]
\item[(ii)] $\mathbf{G}_u$ the integer rectifiable $k$-dimensional current 
\[
\mathbf{G}_u := \mathbf{G}_{u^+}\res E_+\times \R^n - \mathbf{G}_{u^-} \res E_-\times \R^n + Q\,
\mathbf{G}_{\etab\circ u} \res E_0 \times \mathbb R^n\, .
\]
\end{itemize}
\end{definition}

\begin{remark}
Even though \cite{DLS_Currents} only defines multi-valued push-forwards and graphs over a Lipschitz $k$-dimensional submanifold, the theory can be easily extended to treat the case when the domain of the map is a countably $k$-rectifiable set; see \cite{SS17a} for details.

It is also not difficult to see that, if $E$ is closed, then $\spt (\mathbf{G}_u) \subset \gr (u)$. In fact, under some additional assumptions, for instance
when $E$ is a compact Lipschitz submanifold, we can easily conclude that $\spt (\mathbf{G}_u) = \gr (u)$. 
\end{remark}

\begin{lemma}
Let $\Omega \subset \mathbb R^m$ be a bounded Lipschitz domain and $u: \overline{\Omega} \to \Iqspec$ a Lipschitz map. Then, for $p=2Q$, 
\begin{itemize}
\item[(i)] $\partial \mathbf{G}_u = \mathbf{G}_{u|_{\partial \Omega}}\; \modp$;
\item[(ii)] $\mathbf{G}_u$ is a representative $\modp$ (in fact, for every measurable $E\subset \Omega_0$, the current 
$(\mathbf{G}_u - 2Q \mathbf{G}_{\etab\circ u}) \res E\times \R^n$ is also a representative $\modp$). 
\end{itemize}
Moreover, there are positive geometric constants $c(m,n,Q)$ and $C(m,n,Q)$ such that, if $E\subset \Omega$ is Borel measurable and $\Lip (u) \leq c$, then
\begin{equation}\label{e:Taylor_exp}
\left|\|\mathbf{G}_u\| (E\times \mathbb R^n) - Q |E| - {\textstyle{\frac{1}{2}}}\Dir (u, E)\right| \leq C \int_E |Du|^4\, .
\end{equation}
\end{lemma}
\begin{proof}
Recall that, by \cite{Federer69}, an integer rectifiable current $T$ is a representative $\modp$ if and only if its density is at most $\frac{p}{2}$ 
at $\|T\|$-a.e. point. Since this is obviously the case for the current $(\mathbf{G}_u - 2Q \mathbf{G}_{\etab\circ u} )\res E\times \R^n$ for every measurable subset $E\subset \Omega_0$, the second point is trivial. Observe that
\begin{align*}
\mathbf{G}_{u^+} &= \mathbf{G}_{u^+}\res  \Omega_+\times \R^n + Q\mathbf{G}_{\etab\circ u} \res (\Omega_0\cup \Omega_-) \times \R^n\\
\mathbf{G}_{u^-} &= \mathbf{G}_{u^-}\res \Omega_-\times \R^n + Q\mathbf{G}_{\etab\circ u} \res (\Omega_0\cup \Omega_+) \times \R^n\, .
\end{align*}
Therefore we conclude
\begin{align*}
\mathbf{G}_u &= \mathbf{G}_{u^+} - \mathbf{G}_{u^-} + Q \mathbf{G}_{\etab\circ u} - 2Q \mathbf{G}_{\etab\circ u} \res \Omega_-\times \R^n\, .
\end{align*}
In particular 
\[
\mathbf{G}_u =  \mathbf{G}_{u^+} - \mathbf{G}_{u^-} + Q \mathbf{G}_{\etab\circ u}\, \modp\, .
\]
Furthermore, by applying the boundary operator $\modp$ to the above equation we see that
\[
\partial \mathbf{G}_u = \partial \mathbf{G}_{u^+} - \partial \mathbf{G}_{u^-} + Q \partial \mathbf{G}_{\etab\circ u}\, \modp\, .
\]
We can now use the relation $\partial \mathbf{G}_f = \mathbf{G}_{f|_{\partial \Omega}}$ valid for single valued and multivalued Lipschitz graphs (cf. \cite{DLS_Currents}) to conclude
\[
\partial \mathbf{G}_u =  \mathbf{G}_{u^+|_{\partial \Omega}} - \mathbf{G}_{u^-|_{\partial \Omega}} + Q \mathbf{G}_{\etab\circ u|_{\partial \Omega}}\, \modp\, .
\]
Now, using the same argument above we get as well
\[
 \mathbf{G}_{u^+|_{\partial \Omega}} - \mathbf{G}_{u^-|_{\partial \Omega}} + Q \mathbf{G}_{\etab\circ u|_{\partial \Omega}}
- 2Q \mathbf{G}_{\etab\circ u|_{\partial \Omega}} \res (\partial \Omega)_- \times \R^n = \mathbf{G}_{u|_{\partial \Omega}}\, ,
\]
hence concluding the proof of the first point.

We now come to \eqref{e:Taylor_exp}.
First of all, by the obvious additivity in the set $E$ of the various quantities involved in the inequality, it suffices to show it for
subsets $E$ of, respectively, $\Omega_+$, $\Omega_-$ and $\Omega_0$. For subsets of $\Omega_0$ the inequality is the standard Taylor expansion of the area functional for
Lipschitz graphs. 
Next, recall that, by \cite[Corollary 3.3]{DLS_Currents}, the inequality in \eqref{e:Taylor_exp} holds for $\mathbf{G}_{u^+}$ and $\mathbf{G}_{u^-}$ 
(in fact, note that \cite[Corollary 3.3]{DLS_Currents} is stated for Lipschitz open domains $E$, rather than for Borel sets $E$; however, since for any Borel set we can find a sequence $E_k\supset E$ of Lipschitz open domains with $|E_k\setminus E|\to 0$, it is straightforward to infer the validity of
\cite[Corollary 3.3]{DLS_Currents} for a general Borel $E$). If we take $E\subset \Omega_+$, from \cite[Corollary 3.3]{DLS_Currents} and Proposition
\ref{p:app_differentiability} we then immediately conclude
\begin{align*}
 \left|\|\mathbf{G}_u\| (E\times \mathbb R^n) - Q |E| - {\textstyle{\frac{1}{2}}}\Dir (u, E)\right|
 & = \left|\|\mathbf{G}_{u^+}\| (E\times \mathbb R^n) - Q |E| - {\textstyle{\frac{1}{2}}}\Dir (u^+, E)\right|\\
& \leq C \int_E |Du^+|^4 =  C \int_E |Du|^4\, .
\end{align*}
The case $E\subset \Omega_-$ can be proved in a similar fashion since 
\[
\|- \mathbf{G}_{u^-}\| ( E \times \mathbb R^n) = \|\mathbf{G}_{u^-}\| (E\times \mathbb R^n)\, . \qedhere
\]
\end{proof}

\section{BiLipschitz embeddings and retractions, Lipschitz extensions}

In this section we show that, as it is the case for $\Iqs$, there are a suitable biLipschitz embedding of $\Iqspec$ into a sufficiently large Euclidean space and a corresponding retraction map of the ambient onto the embedding. 

\begin{theorem}\label{thm:embedding_and_retraction}
For every $Q$ and $n$ there is $\bar N (n,Q)$ and constants $C(n,Q), \delta_0 (n, Q) >0 $ with the following properties. 
\begin{itemize}
\item[(i)] There is an injective map $\zetab: \Iqspec \to \mathbb R^{\bar N}$ such that
\begin{itemize}
\item[(a)] $\Lip (\zetab), \Lip (\zetab^{-1}) \leq C$, where $\zetab^{-1}$ denotes the inverse of $\zetab$ on $\mathpzc{Q} := \zetab (\Iqspec)$;
\item[(b)] $\Dir (u, M) = \int_{M} |D (\zetab \circ u)|^2$ for every Lipschitz submanifold $M$ of any Euclidean space and for
every $u\in W^{1,2} (M, \Iqspec)$;
\item[(c)] $|\zetab (P)| = |P|$ for every $P\in \Iqspec$. 
\end{itemize}
\item[(ii)] There is a map $\varrhob: \mathbb R^{\bar N} \to \mathpzc{Q}$ with $\Lip (\varrhob)\leq C$ and $\varrhob (x) = x$
for every $x\in \mathpzc{Q}$.\
\item[(iii)] For every positive $\delta < \delta_0$ there is a map $\varrhob^\star_\delta: \mathbb R^{\bar N} \to \mathpzc{Q}$ such that
$|\varrhob_\delta^\star (P)-P|\leq C \delta^{8^{-nQ}}$ for every $P\in \mathpzc{Q}$ and such that the following estimate holds for
every $u\in W^{1,2} (M, \mathbb R^{\bar N})$:
\begin{equation}\label{e:est_rho_star}
\int_M |D (\varrhob^\star_\delta \circ u)|^2 \leq \left(1 +C \delta^{8^{-nQ-1}}\right) \int_{\{\dist (u, \mathpzc{Q}) \leq \delta^{nQ+1}\}}
|Du|^2 + C \int_{\{\dist (u, \mathpzc{Q}) > \delta^{nQ+1}\}} |Du|^2\, .
\end{equation}
\end{itemize}
\end{theorem}

\begin{remark}\label{r:average_zero}
Observe that, in the proof given below, if we identify $\Iqspec$ with $\IQSn\times \mathbb R^n$, then:
\begin{itemize}
\item the map $\zetab$ takes the form $\zetab (P, v) = (\zetab_0 (P), v)$ for a suitable
$\zetab_0:\IQSn \to \mathbb R^{\bar N-n}$;
\item the map $\varrhob$ takes the form $(w,v) \mapsto (\varrhob_0 (w), v)$ for a $\varrhob_0 : \mathbb R^{\bar N-n} \to  \zetab_0 (\IQSn)$;
\item the map $\varrhob^\star_\delta$ takes the form $(w,v) \mapsto (\varrhob_{0, \delta}^\star (w), v)$ for a $\varrhob_{0, \delta}^\star : \mathbb R^{\bar N-n} \to  \zetab_0 (\IQSn)$. 
\end{itemize}
Clearly the maps $\zetab_0, \varrhob_0$ and $\varrhob_{0, \delta}^\star$ enjoy all the properties and estimates claimed in Theorem \ref{thm:embedding_and_retraction} with $\IQSn$ replacing $\Iqspec$ and $\mathbb R^{\bar N -n}$ replacing $\mathbb R^{\bar N}$. 
\end{remark}

\begin{proof}
In the whole proof we identify $\Iqspec$ with 
\[
((\Iqsn\times \{Q\a{0}\}) \cup (\{Q\a{0}\} \times \Iqsn))\times \mathbb R^n \subset \Iqs \times \Iqs \times \mathbb R^n\, .
\]

\medskip

{\bf{Proof of (i).}} Consider the restriction of the map $\xib_{BW}$ of
\cite[Corollary 2.2]{DLS_Qvfr} to $\Iqsn$, which takes values in $\mathbb R^N$ for some $N= N (Q,n)$, and denote by ${\rm id}: \mathbb R^n \to \mathbb R^n$ the identity map. We then see that (a) and (b) for $\zetab = \xib_{BW}\times \xib_{BW} \times {\rm id}$ follow directly from 
\cite[Corollary 2.2]{DLS_Qvfr} and the fact that $\mathcal{G}_s = \mathcal{G}\times \mathcal{G} \times d$ with $d$ as in Proposition \ref{prop.isometry}. For point (c) we need the
fact that $|\xib_{BW} (P)|=|P|$ for every $P\in \Iqs$: although this is not claimed in the statement of \cite[Corollary 2.2]{DLS_Qvfr}, it follows easily
from \cite[Eq. (2.1)]{DLS_Qvfr} and the fact that 
\begin{equation}\label{e:conical}
\xib_{BW} (\sum_i \a{\lambda P_i}) = \lambda \xib_{BW} (\sum_i \a{P_i})\, ,
\end{equation} 
which in turn is
an obvious outcome of the definition of $\xib_{BW}$ given in \cite[Section 2.1.3]{DLS_Qvfr}. 

\medskip

{\bf{Proof of (ii).}} We would like to define the map $\varrhob$ as $\rhob \times \rhob \times {\rm id}$, where $\rhob$ is the map of \cite[Theorem 2.1]{DLS_Qvfr}. Note that the $\xib$ of \cite[Theorem 2.1]{DLS_Qvfr} can be taken to be $\xib_{BW}$, as it is obvious from the discussion in \cite[Section 2.1]{DLS_Qvfr}). In order to simplify the notation, from now on we drop the subscript $_{BW}$.

The first issue is that $\rhob$ is a retraction of $\mathbb R^N$ onto $\mathcal{Q} = \xib (\Iqs)$ rather than onto $\xib (\Iqsn)$. 
In order to deal with it, take $r: \Iqs \to \Iqsn$ as $r(P) := P \ominus \etab (P)$ and substitute $\rhob$ with $\rhob' := \xib \circ r \circ \xib^{-1} \circ \rhob$. The second issue is that $\rhob'\times \rhob'$ is a retraction of $\mathbb R^N\times \mathbb R^N$ onto $\xib (\Iqsn)\times \xib (\Iqsn)$, so that our next goal is to find a retraction of $\xib (\Iqsn) \times \xib (\Iqsn)$ onto ${\xib (\Iqsn)\times \{0\} \cup \{0\} \times \xib (\Iqsn)}$. 
We first define $R : \mathbb R^N \times \mathbb R^N \to \R^N\times \R^N$ as
\[
R (x,y) := 
\begin{cases}
\left(x - \frac{|y|}{|x|} x , 0\right) \quad &\mbox{if $|x|> |y|$}\\
\left(0, y - \frac{|x|}{|y|} y\right) \quad &\mbox{if $|y|> |x|$}\\
(0,0) &\mbox{if $|y|=|x|$.}
\end{cases}
\]
Clearly $R$ maps $\R^N\times \R^N$ onto $\R^N\times \{0\} \cup \{0\}\times \mathbb R^N$ and it is the identity on $\R^N\times \{0\} \cup \{0\}\times \mathbb R^N$. It can be checked in an elementary way that $R$ is Lipschitz. A quick method to see it is the following. First observe that $R$ is obviously locally Lipschitz on $(\mathbb R^N \times \mathbb R^N)\setminus \{(0,0)\}$. By Rademacher's theorem we can compute its differential, which we can do separately on the two relevant open regions $\{|x| >|y|\}$ and $\{|y| > |x|\}$. On the first region the differential is 
\[
DR = \left(\begin{array}{ll}
A & B\\
0 & 0
\end{array}\right)
\] 
where
\begin{align*}
A &= \left(1-\frac{|y|}{|x|}\right) {\rm Id} + \frac{|y|}{|x|^3} x\otimes x\, \\
B &= -\frac{1}{|	x||y|} x \otimes y\, .
\end{align*}
Using the fact that $|y|<|x|$, we easily estimate the operator norm of the differential by $\||DR\|_o \leq \sqrt{2}$. Similarly in the region $\{|y|> |x|\}$. 
We have just concluded that the map $R$ is locally Lipschitz with constant $\sqrt{2}$ on the open set $\{|y|\neq |x|\}$. Since it is continuous and it is constant on the closed set $\{|y|=|x|\}$, it is elementary to see that it is globally Lipschitz with constant $\sqrt{2}$. 

Now, observe that, by \eqref{e:conical}, $R$ maps $\xib (\Iqsn) \times \xib (\Iqsn)$ into $\xib (\Iqsn) \times \xib (\Iqsn)$, and hence into
\[
\begin{split}
(\xib (\Iqsn) \times \xib (\Iqsn)) & \cap (\R^N\times \{0\} \cup  \{0\}\times \R^N) \\ 
&= \xib (\Iqsn)\times \{0\} \cup \{0\} \times \xib (\Iqsn)\, .
\end{split}
\]
We can thus finally define our map $\varrhob$ as $\varrhob = (R \circ (\rhob'\times \rhob'))\times {\rm id}$.  

\medskip

{\bf{Proof of (iii).}} We first consider the map $\rhob_\delta^\star$ of \cite[Proposition 7.2]{DLS_Lp}. As above, a first candidate for the map\
$\varrhob_\delta^\star$ would be $\rhob_\delta^\star \times \rhob_\delta^\star \times {\rm id}$. Again we start replacing $\rhob_\delta^\star$ with $\rhob'_\delta := \xib \circ r \circ \xib^{-1} \circ \rhob_\delta^\star$. Fix $P\in \xib(\Iqsn)$. Recall that, by \cite[Proposition 7.2]{DLS_Lp}, $|\rhob^\star_\delta (P)-P|\leq C \delta^{8^{-nQ}}$. Next, 
\begin{align*}
|\rhob'_\delta (P) - P| & \leq C \mathcal{G} (r (\xib^{-1} (\rhob^\star_\delta (P)) , \xib^{-1} (P))\\
&\leq C \left(\mathcal{G} (\xib^{-1} (\rhob^\star_\delta (P)), \xib^{-1} (P)) + \sqrt{Q} |\etab (\xib^{-1} (\rhob^\star_\delta (P))|\right)\\
&=  C \left(\mathcal{G} (\xib^{-1} (\rhob^\star_\delta (P)), \xib^{-1} (P)) + \sqrt{Q} |\etab (\xib^{-1} (\rhob^\star_\delta (P))- \etab (\xib^{-1} (P))|\right)\\
&\leq 2 C \mathcal{G} (\xib^{-1} (\rhob^\star_\delta (P)), \xib^{-1} (P))\leq 2 C^2 |\rhob^\star_\delta (P) - P|\, .
\end{align*}
We thus conclude the estimate
\begin{equation}\label{e:intermedia}
|\rhob'_\delta (P) - P| \leq C \delta^{8^{-nQ}} \qquad \forall P\in \xib (\Iqsn)\, .
\end{equation}
Furthermore, recall the elementary observation that $\Dir (f\ominus (\etab\circ f)) \leq \Dir (f)$, valid for every $f\in W^{1,2} (\Omega, \Iqs)$. In particular, combining it with \cite[Proposition 7.2]{DLS_Lp} and with part (i) of the theorem, we achieve
\begin{align}
\int_M |D (\rhob'_\delta \circ f)|^2 &\leq \left(1 +C \delta^{8^{-nQ-1}}\right) \int_{\{\dist (f, \xib (\Iqsn)) \leq \delta^{nQ+1}\}}
|Df|^2\nonumber\\ 
&\qquad + C \int_{\{\dist (f, \xib (\Iqsn)) > \delta^{nQ+1}\}} |Df|^2\label{e:intermedia2}
\end{align}
for every $f\in W^{1,2} (M, \R^n)$. 

Our map $\varrhob^\star_\delta$ will be defined as $(R_\delta \circ(\rhob'_\delta \times \rhob'_\delta))\times {\rm id}$, where
\[
R_\delta: \mathbb R^N\times \mathbb R^N\to (\mathbb R^N\times \{0\}\cup \{0\}\times \mathbb R^N)
\]
is an appropriate ``almost retraction'' map which we will construct as follows. First introduce the function $\chi_\delta :[0, \infty[ \to [0, \infty[$ as 
\[
\chi_\delta (s) = \begin{cases}
0 \qquad &\mbox{if $s\in [0, \delta]$}\\
(1-\delta)^{-1} (s-\delta) &\mbox{if $s\in [\delta, 1]$}\\
1 &\mbox{otherwise.}
\end{cases} 
\]
We then define 
\[
R_\delta (x,y) = \begin{cases} 
\left(\chi_\delta (|x|)\frac{x}{|x|}, 0\right)\qquad &\mbox{if $|y|\leq \delta^2$}\\
\left(0, \chi_\delta (|y|)\frac{y}{|y|}\right)\qquad &\mbox{if $|x|\leq \delta^2$.}
\end{cases}
\] 
It is easy to see that $R_\delta$ is well defined, since on the intersection $\{\max\{|y|, |x|\} \geq \delta^2\}$ the map is
identically $0$. Moreover:
\begin{itemize}
\item the restriction of $R_\delta$ to
$\{|y|\leq \delta^2\}$ takes values into $\mathbb R^N\times \{0\}$ and has Lipschitz constant bounded by $1+C\delta$;
\item the restriction of $R_\delta$ to
$\{|x|\leq \delta^2\}$ takes values into $\{0\}\times \mathbb R^N$ and has also Lipschitz constant bounded by $1+C\delta$.
\end{itemize} 
Its global Lipschitz constant is controlled independently of $\delta$ and, finally, we can extend it to the whole $\R^N\times \R^N$ by first choosing a Lipschitz extension taking values in $\R^N\times \R^N$ and then composing it with the retraction map $R$ of the proof of (ii)

We will now show that $\varrhob^\star_\delta := (R_\delta \circ(\rhob '_\delta \times \rhob'_\delta))\times {\rm id}$ has the desired properties. First observe that, if a point $P = (p,q,v)\in \R^{2N+1}$ belongs to $\mathpzc{Q}$, then either $p=0$ or $q=0$. Without loss of generality, assume that the second alternative holds. Then $R_\delta (p,0) = (p', 0)$ with $|p-p'|\leq C\delta$ and moreover $p'$ is a positive multiple of $p$,
which by \eqref{e:conical} implies that $p'\in \xib (\Iqsn)$. We therefore find that 
\begin{align*}
|\varrhob^\star_\delta (P) - P| & = |\rhob'_\delta (p') - p| \leq  |\rhob'_\delta (p') - p'| + |p'-p| \leq C \delta^{8^{-nQ-1}} + C \delta\, . 
\end{align*}
We next come to \eqref{e:est_rho_star}. Without loss of generality observe that we can prove the estimate for a generic Lipschitz map $u= (v,w,z)$ on a bounded domain. Consider next the set $E:= \{\dist (u, \mathpzc{Q})\leq \delta^{nQ+1}\}$. Let $u = (v,w,z)$ and let 
$(v',w') = R_\delta (v,w)$. If $z\in E$, we then have two cases
\begin{itemize}
\item $w' (z) =0$ and $\dist (v' (z), \xib (\Iqsn)) \leq \dist (u(z),  \mathpzc{Q})\leq \delta^{nQ+1}$;
\item $v'(z) =0$ and $\dist (w' (z), \xib (\Iqsn)) \leq \dist (u(z),  \mathpzc{Q})\leq \delta^{nQ+1}$.
\end{itemize}
In the first case we have $\varrhob^\star_\delta \circ u (x) = (\rhob'_\delta (v'(x)), 0, z(x))$, whereas in the second case we have
$\varrhob^\star_\delta \circ u (x) = (0, \rhob'_\delta (w'(x)), z(x))$. Using \eqref{e:intermedia2} we then can easily estimate
\begin{equation}
\begin{split}
\int_M |D (\varrhob^\star_\delta \circ u)|^2 \leq \left(1+C \delta^{8^{-nQ-1}}\right) &\int_E |D (R_\delta \circ (v,w))|^2 \\
&+ C \int_{M\setminus E} |D (R_\delta \circ (v,w))|^2  + \int_M |Dz|^2\, .
\end{split}
\end{equation}
Observe also that $(v,w) (E)$ is contained in $\{(x,y) : \min\{|x|, |y|\} \leq \delta^{nQ+1}\leq \delta^2\}$. On this set we easily compute $|DR_\delta|\leq 1 +C \delta$. Moreover recall that $\|DR_\delta\|_\infty\leq C$ for some constant $C$ independent of $\delta$. Thus we can write 
\begin{equation}
\begin{split}
\int_M |D (\varrhob^\star_\delta \circ u)|^2 \leq \left(1+C \delta^{8^{-nQ-1}}\right) &\int_E (|Dv|^2 + |Dw|^2) 
\\
&+ C \int_{M\setminus E} (|Dv|^2 + |Dw|^2) + \int_M |Dz|^2\, .
\end{split}
\end{equation}
Considering that $|Du|^2 = |Dv|^2 + |Dw|^2 + |Dz|^2$, we then conclude the desired estimate \eqref{e:est_rho_star}. 
\end{proof}

We conclude this section by remarking that a simple corollary of the parts (i) and (ii) of the above theorem is the following analogue of 
\cite[Theorem 1.7]{DLS_Qvfr}, recorded as Corollary \ref{cor.Lipschtiz extension for Iqs} here below. In turn using the corollary, a simple inspection of the proof of the Lipschitz approximation theorem in \cite[Proposition 2.5]{DLS_Qvfr} shows that the same result is valid for Sobolev maps with values in $\Iqspec$.

\begin{corollary}\label{cor.Lipschtiz extension for Iqs}
Let $B \subset \R^m$ and $f: B \to \Iqspec$ be Lipschitz. Then there exists an extension $\bar{f}: \R^m \to \Iqspec$ of $f$, with $\Lip(\bar{f}) \le C(m,Q)\, \Lip(f)$. Moreover, if $f$ is bounded then
\begin{equation}\label{e:sharp_infty}
\sup_{x \in \R^m} \abs{\bar{f} (x)} \le \sup_{x \in B} \abs{f (x)}\,, 
\end{equation}
and for any $q \in \R^n$ it holds
\begin{equation} \label{e:ext_osc1}
\sup_{x \in \R^m}\cG_s(\bar f(x), Q \a{q}) \leq C(m,Q)\, \sup_{x \in B} \cG_s(f(x), Q \a{q})\,.
\end{equation}
\end{corollary}
\begin{proof}
In order to get the Lipschitz extension it suffices to first extend $\zetab\circ f$ and then compose the extension with $\zetab^{-1}\circ \varrhob$. 
Next, assume that $M := \sup_{x \in B} \abs{f (x)} < \infty$. Observe that $\Iqspec$ is a cone, namely for every $\lambda \in [0, \infty[$ we can define $\lambda (T, \epsilon) = (\sum_i \a{\lambda T_i}, \epsilon)$. We therefore introduce the projection of $\Iqspec$ onto $\{S\in \Iqspec : |S|\leq M\}$ by keeping $S$ fixed if $|S|\leq M$ and mapping it to $\frac{M}{|S|} S$ if $|S|> M$. Such projection is $1$-Lipschitz and it suffices to compose it with any Lipschitz extension of $f$ to obtain a new extension with no larger Lipschitz constant and which satisfies the bound \eqref{e:sharp_infty}. Next, we prove \eqref{e:ext_osc1}. First, let us define the \emph{oscillation} of $f$ by
\begin{equation} \label{e:oscillation}
{\rm osc}(f) := \inf_{q \in \R^n} \sup_{x \in B} \cG_s(f(x), Q\a{q})\,, 
\end{equation}
and observe that since $f$ is bounded the infimum in \eqref{e:oscillation} is achieved. Let us then call $q_0 \in \R^n$ a value which realizes the oscillation, so that
\[
R := {\rm osc}(f) = \sup_{x \in B} \cG_s(f(x), Q\a{q_0})\,.
\]
Of course, if $R = 0$ then $f$ is identically equal to $Q\a{q_0}$ on $B$, and thus \eqref{e:ext_osc1} is trivially true for the natural extension $\bar f(x) = Q\a{q_0}$ for every $x \in \R^m$. Thus, we can assume $R >0$. We also set $L := \Lip(f)$. Then, we introduce the map
\begin{equation} \label{e:mappa fantasma}
 \tilde f \colon (B \times \{0\}) \cup (\R^m \times \{\frac{R}{L}\}) \subset \R^{m+1} \to \Iqspec
\end{equation}
which extends $f$, and which takes value $\tilde f(z) := Q\a{q_0}$ at every point $z = (x,R/L)$ with $x \in \R^m$. Since for any given $(x,0) \in B \times \{0\}$ and $z \in \R^m \times \{\frac{R}{L}\}$ we have
\[
\cG_s(\tilde f((x,0)), \tilde f(z)) = \cG_s(f(x),Q\a{q_0}) \leq R = L\, \left(\frac{R}{L} \right) \leq L \, \abs{(x,0) - z}\,,	
\]
it is clear that $\Lip(\tilde f) = \Lip(f) = L$. We can now use the argument in the first part of the proof to extend $\tilde{f}$ to a function $F \colon \R^{m+1} \to \Iqspec$, and then define $\bar f(x) := F((x,0))$ for all $x \in \R^m$. It is clear that $\bar f$ is an extension of $f$, and that both $\Lip(\bar f) \leq C(m,Q)\, \Lip(f)$ and \eqref{e:sharp_infty} hold. We claim that this extension $\bar f$ also satisfies \eqref{e:ext_osc1}. To this aim, let $q \in \R^n$, and set
\begin{equation} \label{ext_parameter}
\delta_q := \cG_s(Q\a{q}, Q\a{q_0}) = \sqrt{Q} \, \abs{q-q_0}\,.
\end{equation}
We shall distinguish two cases. Set $C = C(m,Q)$ the constant above, and assume first that 
\begin{equation} \label{first case}
\delta_q \leq (C+1) \, R\,.
\end{equation}
Then, for any $x \in \R^m$ it holds
\[
\cG_s(\bar f(x), Q\a{q}) \leq \delta_q + R \leq (C+2)\,R \leq (C+2)\, \sup_{x \in B} \cG_s(f(x),Q\a{q})\,,
\]
where in the last inequality we have used the definition of $R$. This proves the validity of \eqref{e:ext_osc1} when \eqref{first case} holds. Let us then suppose that \eqref{first case} fails, so that
\begin{equation} \label{second case}
(C+1)\,R < \delta_q.
\end{equation}
By triangle inequality we have, for every $x \in B$:
\begin{equation} \label{fatto1}
\cG_s(f(x), Q\a{q}) \geq \delta_q - R \geq \frac{C}{C+1} \, \delta_q \geq \frac{\delta_q}{2}\,.
\end{equation}
On the other hand, for any $y \in \R^m$ it holds
\begin{equation} \label{fatto2}
\cG_s(\bar f(y), Q\a{q_0}) = \cG_s(F((y,0)), F(y,R/L)) \leq C\, R\,,
\end{equation}
so that if we combine \eqref{fatto1} and \eqref{fatto2} we obtain
\begin{equation} \label{chiudo}
\cG_s(\bar f(y), Q\a{q}) \leq \delta_q + C\, R \overset{\eqref{second case}}{\leq} 2\, \delta_q \leq 4\, \cG_s(f(x), Q\a{q})
\end{equation}
for every $x\in B$, for every $y \in \R^m$. This is stronger than \eqref{e:ext_osc1}, and thus it concludes the proof. 
\end{proof}

Finally, we record here another simple consequence of the existence of the embeddings and of the retraction (for which, again, an intrinsic proof in the
spirit of \cite[Section 4.3.1]{DLS_Qvfr} is also possible).

\begin{lemma}[Luckhaus lemma]\label{lem.easyLuckhaus}
There is a constant $C(m,n,Q)$ with the following property. Assume $f,g \in W^{1,2} (\mathbb S^{m-1}, \Iqspec)$ (resp.
$f,g\in W^{1,2} (\mathbb S^{m-1}, \IQSn)$ and let $\lambda <\frac{1}{2}$ be a given positive number. Then there
is a $u\in W^{1,2} (B_1\setminus B_{1-\lambda}, \Iqspec)$ (resp. $u\in W^{1,2} (B_1\setminus B_{1-\lambda}, \IQSn)$) such that
\begin{align}
& u_{\partial B_1} = f \qquad \mbox{and}\qquad u|_{\partial B_{1-\lambda}} = g \left({\textstyle{\frac{\cdot}{1-\lambda}}}\right)\\
& \Dir (u) \leq C \lambda \left(\Dir (f, \mathbb S^{m-1}) + \Dir (g, \mathbb S^{m-1})\right) + C \lambda^{-1} \int_{\mathbb S^{m-1}} \mathcal{G}_s (f,g)^2\, .  
\end{align}
If $f,g$ are, in addition, Lipschitz continuous, then the interpolating function $u$ can be chosen such that
\begin{equation}\label{e:interpolation Lipschitz bound}
	\Lip(u) \le C  \left( \Lip(f) + \Lip(g) \right) + C \lambda^{-1} \norm{\G_s(f,g)}_\infty\, .
\end{equation}
\end{lemma}
\begin{proof}
Consider the case $\Iqspec$. The map $u$ can be explicitly defined via
\[
u (x) = (\zetab^{-1} \circ \varrhob ) \left(\frac{|x| - (1-\lambda)}{\lambda} \zetab \left(f \left(\frac{x}{|x|}\right)\right) + \frac{1-|x|}{\lambda} 
\zetab \left(g \left(\frac{x}{|x|}\right)\right)\right)\, .
\]
In the case $\IQSn$ we use the maps $\zetab_0$ and $\varrhob_0$ of Remark \ref{r:average_zero} in place of $\zetab$ and $\varrhob$. 
\end{proof}

Another useful tool will be the following approximation lemma. It is the $\Iqspec$ version of \cite[Lemma 4.5]{DLS_Lp}.
\begin{lemma}\label{l:lip_app}
Let $f$ be a map in $W^{1,2}(B_r, \Iqspec)$, where $B_r \subset \R^m$. Then for every $\varepsilon$ there
exists an approximating map $f_\varepsilon \in W^{1,2}(B_r, \Iqspec)$ such that
\begin{itemize}
\item[(a)] $f_\varepsilon$ is Lipschitz continuous;
\item[(b)] The following estimate holds:
\begin{equation}\label{e:lip_smoothing}
\int_{B_r}\G_s(f,f_\varepsilon)^2+\int_{B_r}\big(|Df|-|Df_\varepsilon|\big)^2
+ \int_{B_r} |D (\etab\circ f)- D(\etab\circ f_\varepsilon)|^2
\leq \varepsilon.
\end{equation}
\end{itemize}
If $f\vert_{\partial B_r}\in W^{1,2}(\partial B_r,\Iqspec)$,
then $f_\epsilon$ can be chosen to satisfy also
\begin{equation}\label{e:approx bordo}
\int_{\partial B_r}\G(f,f_\varepsilon)^2+\int_{\partial B_r}\big(|Df|-|Df_\varepsilon|\big)^2 \leq \varepsilon.
\end{equation}
\end{lemma}
The proof is the very same as given in \cite[Lemma 4.5]{DLS_Lp}, only using the Lipschitz extension theorem for $\Iqspec$.

\section{Existence and compactness of $\Dir$-minimizers}

The following existence theorem is a simple consequence of the fact that we can identify $\Iqspec$ (resp. $\IQSn$) with a closed subset
of $\Iqs\times \Iqs \times \mathbb R^n$ (resp. $\Iqs \times \Iqs$), and that the Dirichlet energy of an $\Iqspec$-valued map is the sum of the Dirichlet energies of the corresponding factors (with the Dirichlet energy of the center of mass weightd by $Q$); see Definition \ref{d:Sobolev}. Therefore we leave the proof to the reader.

\begin{theorem}\label{thm:existence}
Assume $\Omega \subset \mathbb R^m$ is a bounded Lipschitz set and let $f\in W^{1,2} (\Omega, \Iqspec)$ (resp. $f\in W^{1,2} (\Omega, \IQSn)$. Then there is a map $g\in W^{1,2} (\Omega, \Iqspec)$ (resp. $g\in W^{1,2} (\Omega, \IQSn)$) such that $(g-f)|_{\partial \Omega} =0$ (in the
sense of the trace theorem \cite[Proposition 2.10]{DLS_Qvfr}) and which minimizes the Dirichlet energy over all maps with the same trace property. 
\end{theorem}

Note that if $f(x) = (\tilde{f}(x) , 1)$ for a.e. $x \in \Omega$ then $g(x)=(\tilde{g}(x),1)$ for a.e. $x$ and $\tilde{g}$ is minimizing in  $W^{1,2}(\Omega, \Iqs)$.

\begin{definition}\label{d:minimizing}
A map $g$ as in Theorem \ref{thm:existence} will be called a $\Dir$-minimizer (or $\Dir$-minimizing) in $\Omega$. 
\end{definition}

Moreover, the following is another obvious consequence of the ``factorization'' of $\Iqspec$ into $\IQSn\times \R^n$,
in particular of \eqref{e:split_energy}. 

\begin{proposition}\label{p:factorization}
A map $u\in W^{1,2} (\Omega, \Iqspec)$ is $\Dir$-minimizing in $\Omega$ if and only if both $u\ominus \etab\circ u$ and $\etab\circ u$ are 
$\Dir$-minimizing in $\Omega$. Moreover $u\ominus \etab\circ u$ is a $\Dir$-minimizer in $W^{1,2} (\Omega, \IQSn)$ if and only if it is a 
$\Dir$-minimizer in $\Iqspec$. 
\end{proposition} 

We close this section by the following compactness property of $\Dir$-minimizers

\begin{proposition}\label{prop.compactness of minimizers}
Let $\{g_k\}\subset W^{1,2} (B_r, \Iqspec)$ be a sequence of maps which are $\Dir$-minimizing in $B_r$ and which converge weakly to some $g$. Then, for every $s<r$, the sequence converges strongly in $W^{1,2} (B_s, \Iqspec)$ and moreover the limiting $g$ is $\Dir$-minimizing in $B_s$. If $\limsup \Dir (g_k|_{\partial B_r }) < \infty$, then the same conclusion holds in $B_r $. 
\end{proposition}
\begin{proof} First of all,
using Fatou's lemma we get
\[
\int_s^r \liminf_{k\to \infty} \Dir (g_k|_{\partial B_\sigma })\, d\sigma < \infty
\]
and thus we can reduce the first statement to the second. We assume therefore, without loss of generality, that $r=1$ and
\begin{equation}\label{e:aggiuntiva_bordo}
\sup_k \Dir (g_k|_{\partial B_1}) < \infty\, .
\end{equation}
Observe next that, by weak convergence and trace theorems in the Sobolev spaces, we know:
\begin{align}
& \lim_{k\to\infty} \int_{B_1} \mathcal{G}_s (g_k, g)^2 = 0\\
& \lim_{k\to \infty} \int_{\partial B_1} \mathcal{G}_s (g_k, g)^2 = 0\\
&\liminf_{k\to\infty} \Dir (g_k, B_1) \geq \Dir (g, B_1)\\
&\liminf_{k\to\infty} \Dir (g_k|_{\partial B_1}, \mathbb S^{m-1}) \geq \Dir (g|_{\partial B_1}, \mathbb S^{m-1})\, .
\end{align}
Given any $\lambda \in ]0, \frac{1}{2}]$, we can thus apply the Luckhaus Lemma \ref{lem.easyLuckhaus} to find a sequence
of maps $h_k$ on $B_1\setminus B_{1-\lambda}$ such that
\begin{itemize}
\item $h_k (x) = g_k (x)$ for every $x\in \partial B_1$ and $h_k (x) = g (\frac{x}{1-\lambda})$ for any $x\in \partial B_{1-\lambda}$;
\item the following estimate holds
\begin{align}
&\limsup_{k\to \infty} \Dir (h_k, B_1\setminus \overline{B}_{1-\lambda}) \leq C \lambda K\, ,\label{e:perdita}
\end{align}
where $C$ is a geometric constant depending on $m,n,Q$ and 
\[
K = \limsup_k \Dir (g_k|_{\partial B_1}, \mathbb S^{m-1})\, .
\]
\end{itemize}
Assume now by contradiction that either $g$ is not $\Dir$-minimizing or that 
\[
\Dir (g, B_1) < \limsup_{k\to \infty} \Dir (g_k, B_1)\, .
\]
For a subsequence of $\{g_k\}$, not relabeled, we then have that there is a map $\hat{g}$ with $\hat{g}|_{\partial B_1} = 
g|_{\partial B_1}$ and
\begin{equation}\label{e:guadagno}
\lim_{k\to \infty} (\Dir (g_k, B_1) - \Dir (\hat{g}, B_1)) = L > 0\, . 
\end{equation}
Consider then the function 
\[
\hat{g}_k (x) =
\begin{cases}
h_k (x) \qquad &\mbox{if $1-\lambda < |x| < 1$}\\
\hat{g} \left({\textstyle{\frac{x}{1-\lambda}}}\right) \qquad &\mbox{if $|x|\leq 1-\lambda$.}
\end{cases}
\]
Since $\Dir (\hat{g}_k, B_{1-\lambda}) = (1-\lambda)^{m-2} \Dir (\hat g, B_1)\leq \Dir (\hat{g}, B_1)$, combining
\eqref{e:guadagno} and \eqref{e:perdita} we achieve
\[
\liminf_{k\to \infty} (\Dir (g_k, B_1) - \Dir (\hat{g}_k, B_1) \geq L - CK \lambda\, .
\]
In particular, the right hand side of the last inequality can be made positive by choosing $\lambda$ appropriately small.
Since however $\hat{g}_k|_{\partial B_1} = g_k|_{\partial B_1}$, for $k$ large enough we would contradict the minimality of
$g_k$. 
\end{proof}

\section{First variations}

\subsection{Notation for $\Iqspec$-calculus}\label{s:notation_spec}
In this section we derive some key identities for $\Dir$-minimizers $u$ defined over a bounded domain $\Omega$, which come from computing 
first variations of the functional.
We distinguish two types of variations: inner variations and outer variations. Given the decomposition of $\Omega$ as in Definition \ref{d:pieces} in each of the
domains $\Omega_+$, $\Omega_-$ and $\Omega_0$, we can regard $u$ as an $\Iqs$-valued function, coinciding respectively with $u^+$, $u^-$ and $Q \a{\etab\circ u}$. By Proposition \ref{p:app_differentiability}, in each of these
domains the respective map is approximately differentiable and we can use the chain rules of 
\cite[Proposition 2.8]{DLS_Qvfr}. When we deal with integrals over the whole domain we would then have rather cumbersome formulas where we break the integral in the respective domains $\Omega_+$, $\Omega_-$ and $\Omega_0$, in spite of the fact that such formulas would nonetheless be rather straightforward. In order to simplify our notation we will then use the convention that $\sum_i \a{u_i (x)}$, resp. $\sum_i \a{Du_i (x)}$, will denote the multivalued maps
$\sum_i \a{u_i^+ (x)}$, $\sum_i \a{u_i^- (x)}$ and $Q\a{\etab \circ u (x)}$  (resp. $\sum_i \a{Du_i^+ (x)}$, 
$\sum_i \a{Du_i^- (x)}$ and $Q\a{D (\etab \circ u) (x)}$) depending on whether $x$ belongs to $\Omega_+$, $\Omega_-$ or $\Omega_0$. 

\subsection{Inner variations} Inner variations are generated by composing $u$ with one-parameter families of diffeomorphisms $\Phi_t$ of $\Omega$ which are
the identity on $\partial \Omega$. More specifically we consider a vector field $\varphi \in C^\infty_c (\Omega, \mathbb R^m)$, we let $\Phi_t (x) = x + t \varphi (x)$ and we
observe that,  whenever $|t|$ is sufficiently small, $u\circ \Phi_t$ is well defined, $u\circ \Phi_t \in W^{1,2}(\Omega,\Iqspec)$ and
$u\circ \Phi_t |_{\partial \Omega} = u|_{\partial \Omega}$. We therefore conclude
that, if $u$ is $\Dir$-minimizing, then $\Dir (u\circ \Phi_t) \geq \Dir (u)$ for all sufficiently small $t$, and thus
\begin{equation}\label{e:IV}
0 = \left. \frac{d}{dt}\right|_{t=0} \Dir (u\circ \Phi_t)\, .
\end{equation}

Using the discussion above we can break the domain $\Omega$ into the pieces $\Omega_+$, $\Omega_-$ and $\Omega_0$ 
where we use the chain rules of \cite[Proposition 2.8]{DLS_Qvfr} to prove the following proposition (which corresponds to the first part of \cite[Proposition 3.1]{DLS_Qvfr}). Note that, since $\Phi_t$ is a diffeomorphism, the partition of the domain $\Omega$ induced by the map $u \circ \Phi_{t}$ is given by $\{ \Phi_{t}^{-1}(\Omega_+), \Phi_{t}^{-1}(\Omega_-), \Phi_{t}^{-1}(\Omega_0)  \}$.
 
\begin{proposition}\label{p:inner_variations}
Let $\Omega$ be a bounded open set and $u\in W^{1,2} (\Omega,\Iqspec)$ a $\Dir$-minimizer. Then for every $\varphi\in C^\infty_c (\Omega, \mathbb R^m)$ we have
\begin{equation}\label{e:inner_variation}
\int_\Omega \left(2\,\sum_i \langle Du_i : Du_i \cdot D\varphi \rangle - |Du|^2 {\rm div}\, \varphi \right) = 0\, , 
\end{equation}
where $\langle A:B\rangle$ denotes the Hilbert-Schmidt scalar product between $n\times m$ matrices (i.e. $\langle A : B \rangle = \sum_{i,j} A_{ij} B_{ij}$). 
\end{proposition}

\subsection{Outer variations}\label{s:outer} Next consider a map $\psi \in C^\infty (\Omega\times \R^n, \R^n)$ such that $\psi (x,u)=0$ in a neighborhood of $\partial \Omega \times \mathbb R^n$ and which satisfies the growth conditions
\begin{equation}\label{e:growth}
|D_u \psi|\leq C < \infty \qquad \mbox{and}\qquad |\psi (x,u)| + |D_x \psi (x,u)| \leq C (1+|u|)
\end{equation}  
for some constant $C$. For each fixed $x$, consider the map 
\[
\Iqs \ni P = \sum_i \a{P_i} \mapsto \Psi (x, P) = \sum_i \a{P_i + \psi(x,P_i)}\, .
\]
Observe that if we consider the obvious induced map on $\Iqs \sqcup \Iqs$, the latter commutes with the equivalence relation defining $\Iqspec$ and thus
induces a corresponding map on $\Iqspec$ through $(P, \epsilon) \to \Psi (x, (P, \epsilon)) := (\Psi (x,P), \epsilon)$. 

Hence if $u$ takes values in $\Iqspec$, then we have a well defined map $x\mapsto \Psi (x, u(x))$, which we will denote by $\Psi (x,u) = u + \psi(x,u)$ and which in a neighborhood of $\partial \Omega$ agrees with 
$u$. We wish to show that $x\mapsto \Psi (x,u)\in W^{1,2} (\Omega,\Iqspec)$ when $u\in W^{1,2} (\Omega,\Iqspec)$ (and $\Omega$ is bounded). A possible procedure is the following: 
\begin{itemize}
\item When $u$ is Lipschitz, we observe that $\Psi (x,u)$ is also Lipschitz. Using Definition \ref{d:pieces} consider the sets
$\Omega_+$, $\Omega_-$ and $\Omega_0$ and observe that $(\Psi (x, u))^+ (x) = \Psi (x, u^+ (x))$ and $(\Psi (x, u))^- (x) = \Psi (x, u^- (x))$. In particular
\begin{equation}\label{e:composition_formula}
\Dir (\Psi (x,u)) = \int_{\Omega_+} |D \Psi (x, u^+ (x))|^2 + \int_{\Omega_-} |D \Psi (x, u^- (x))|^2 + Q \int_{\Omega_0} |D \Psi (x, \etab \circ u (x))|^2\, . 
\end{equation}
\item Using the chain rules of \cite[Proposition 1.12]{DLS_Qvfr} we see then easily that there is a constant $\hat C$ (depending only on
$m,n, Q$ and $C$ in \eqref{e:growth}) such that, if $u$ is Lipschitz, then
\[
\Dir (\Psi(x,u)) \leq \hat C \left(|\Omega| + \|u\|_{L^2} + \Dir (u)\right)\, .
\]
\item Using the analogue of \cite[Proposition 2.5]{DLS_Qvfr}, for a general map $u\in W^{1,2}$ we find a sequence of Lipschitz maps $u_k$ converging to $u$ in $L^2$ and with equibounded Dirichlet energy. The corresponding maps $x\mapsto \Psi (x, u_k)$ converge then to $\Psi (x,u)$ and have equibounded Dirichlet energy. We then conclude that $\Psi (x,u)\in W^{1,2}$. Next, considering \cite[Equation (2.9)]{DLS_Qvfr}, we can also observe that
the convergence is in fact strong in $W^{1,2}$ and thus \eqref{e:composition_formula} holds for a general $W^{1,2}$ map. 
\end{itemize}

We are now ready to define outer variations. Consider indeed a smooth $\psi$ which is supported in $\Omega'\times \mathbb R^n$ for some $\Omega'\ssubset \Omega$ and has the same properties and growth conditions as above and let 
$\Psi_t (x, u):= u + t\, \psi (x, u)$. Then, if $u$ is a $\Dir$-minimizer, $\Dir (\Psi_t (x,u))\geq \Dir (u)$ and we thus can write
\begin{equation}\label{e:OV}
0 = \left. \frac{d}{dt}\right|_{t=0} \Dir (\Psi_t (x,u))\, .
\end{equation}
Arguing as in the previous paragraph we then conclude the following analogue of the second part of \cite[Proposition 3.1]{DLS_Qvfr}.

\begin{proposition}\label{p:outer_variation}
Let $u\in W^{1,2} (\Omega, \Iqspec)$ be a $\Dir$-minimizer and assume $\psi \in C^\infty (\Omega\times \mathbb R^n, \R^n)$ is a vector field which vanishes identically in a neighborhood of $\partial \Omega \times \R^n$ and satisfies the growth conditions \eqref{e:growth} for some constant $C$. Then we have
\begin{equation}\label{e:outer_variation}
\int \sum_i \left( \langle Du_i : D_x\psi (x, u_i) \rangle + \langle Du_i : D_u \psi (x, u_i) \cdot Du_i \rangle\right) = 0\, .
\end{equation}
\end{proposition}

\subsection{Key identities} Arguing as in \cite[Section 3.1.2]{DLS_Qvfr} we test the identities \eqref{e:inner_variation}
and \eqref{e:outer_variation} with $\varphi$ and $\psi$ of the following special form: $\varphi (x) = \phi (|x|) x$ and
$\psi (x,u) = \phi (|x|) u$. If we let $\phi$ converge to the indicator function of the interval $[0,r[$ we then reach the following key identities.

\begin{proposition}\label{p:key_id}
Let $\Omega\subset \mathbb R^m$ be a bounded open set and let $u\in W^{1,2} (\Omega, \Iqspec)$ be a $\Dir$-minimizer. 
Then for a.e. $r\in [0, \dist (x, \partial \Omega)[$ the following equalities hold:
\begin{align}
(m-2) \int_{B_r (x)} |Du|^2 &= r\int_{\partial B_r (x)} |Du|^2 - 2 r\int_{\partial B_r (x)} |\partial_\nu u|^2\, ,\label{e:inner2}\\
\int_{B_r (x)} |Du|^2 &= \int_{\partial B_r (x)} \sum_i \langle \partial_\nu u_i, u_i\rangle\, , \label{e:outer} 
\end{align}
where $\nu$ denotes the outer unit normal to $\partial B_r (x)$ and $\sum_i \a{\partial_\nu u_i}$ is the multivalued map
$y\mapsto \sum_i \a{Du_i (u) \cdot \nu (y)}$.  
\end{proposition}

The proof of the proposition follows from the very same computations of \cite[Section 3.1.2]{DLS_Qvfr}.

\section{H\"older regularity of $\Dir$-minimizers}

In this section we show that $\Dir$-minimizers are H\"older continuous. In particular we will prove the following

\begin{theorem}\label{thm.interior Hoelder regularity}
There are constants $\alpha_0 (m,n,Q)>0$ and $C(m,n,Q)$ with the following property. Assume $u\in W^{1,2} (B_{2r} (x), \Iqspec)$ is a $\Dir$-minimizer. Then $u\in C^{0, \alpha_0}_{loc} (B_{2r}(x))$. Indeed we have the estimates
\begin{align}
& [u]_{\alpha_0, B_r (x)} \leq C r^{1-m/2 - \alpha_0} \left(\Dir (u, B_{2r} (x)\right)^{\frac{1}{2}}\label{e:C-alpha-est}\\
&\rho^{2-m-2\alpha_0} \Dir (u, B_\rho (x)) \leq (2r)^{2-m-2\alpha_0} \Dir (u, B_{2r} (x)) \qquad \forall\, 0<\rho < 2r\, .\label{eq.energy decay}
\end{align}
\end{theorem}

The estimate \eqref{eq.energy decay} gives a corresponding estimate for $\zetab\circ u$ and then \eqref{eq.energy decay} implies
\eqref{e:C-alpha-est} through the classical theory of Campanato spaces, cf. \cite[Proposition 3.7 \& Theorem 2.9]{giusti}. We therefore
focus our attention on \eqref{eq.energy decay}, which is a direct consequence of the following proposition.

\begin{proposition}\label{prop.energy boundary control}
There is a constant $\alpha_0 (m,n,Q) > 0$ such that the following inequality holds for every $u\in W^{1,2} (B_1 (x), \Iqspec)$:
\begin{equation}\label{eq.energy boundary control}
\Dir (u, B_1) \leq \frac{1}{m-2+2\alpha_0} \Dir (u|_{\partial B_1}, \partial B_1)\, .
\end{equation}
\end{proposition}

Indeed, let $u$ be as in Theorem \ref{thm.interior Hoelder regularity} and set
\begin{align*}
I (\rho) &:= \Dir (u, B_\rho (x))\\
J (\rho) &:= \Dir (u|_{\partial B_\rho (x)}, \partial B_\rho (x))\, .
\end{align*}
Notice that, 
\[
I (\rho) = \int_0^\rho \int_{\partial B_\sigma (x)} |Du|^2\, d\sigma \geq \int_0^\rho J (\sigma)\, d\sigma\, .
\]
Moreover, by rescaling and translating, \eqref{eq.energy boundary control} gives
\[
I (\rho) \leq \frac{\rho}{m-2+2\alpha_0} J (\rho)\, .
\]
We thus conclude easily that $(r^{2-m-2\alpha_0} I (r))' \geq 0$, which obviously implies \eqref{eq.energy decay}. 

We split the proof of Proposition \ref{prop.energy boundary control} into two cases depending on the dimension of the domain, namely $m=2$ and $m>2$. In the case $m=2$ it suffices to prove the existence of a constant $C$ such that, if $\tilde{u}\in W^{1,2} (\partial B_1, \Iqspec)$, then we can find an extension $u$ of $\tilde{u}$ to $B_1$ satisfying the inequality
\[
\Dir (u, B_1) \leq C \Dir (\tilde{u}, \mathbb S^1)\, .
\]
The latter property is a classical fact for usual harmonic extensions of maps with values in the Euclidean space: in that case the constant $C$ can be taken to be $1$. For the $\Iqspec$ case we consider $\zetab \circ \tilde{u}$ and we then let $h$ be its harmonic extension to $B_1$. Setting $u := \zetab^{-1} \circ \varrhob \circ h$, the inequality is then an easy consequence of the estimate for the harmonic extension and the Lipschitz regularity of $\zetab^{-1}$ and $\varrhob$. The case $m\geq 3$ is harder and we need one important auxiliary result.

\subsection{$0$-homogeneous minimizers} The following lemma shows that $0$-homogeneous minimizers are necessarily constant.

\begin{lemma}\label{l:0-lemma}
Let $m\geq 3$ and let $u\in W^{1,2} (B_1, \Iqspec)$ be a $\Dir$-minimizer with the additional property that
\begin{equation}\label{e:0-lemma}
\Dir (u, B_1) \geq \frac{1}{m-2} \Dir (u|_{\partial B_1}, \mathbb S^{m-1})\, .
\end{equation}
Then $u$ is constant. 
\end{lemma}
\begin{proof}
Observe that
\begin{equation}\label{e:boundary_Dir}
\Dir (u|_{\partial B_1}, \mathbb S^{m-1}) = \int_{\partial B_1} \left(|Du|^2 - |\partial_\nu u|^2\right)\, ,
\end{equation}
where $\nu$ denotes the outer unit normal to $\partial B_1$. Using \eqref{e:inner2}, \eqref{e:0-lemma} and \eqref{e:boundary_Dir}
we conclude 
\[
\int_{\partial B_1} |\partial_\nu u|^2 \leq 0\, ,
\]
namely that $\partial_\nu u$ vanishes identically on $\partial B_1$. But then \eqref{e:outer} implies that $\Dir (u, B_1) =0$, which clearly
gives the constancy of the function $u$. 
\end{proof}

\subsection{Proof of Proposition \ref{prop.energy boundary control} for $m \geq 3$} Consider first that, given the classical inequality for Euclidean valued harmonic functions, we can assume without loss of generality that the function $u$ takes values in $\IQSn$. In this case both $\Dir (u, B_1)$ and $\Dir (u|_{\partial B_1}, \partial B_1)$ can be split as 
\begin{align*}
\Dir (u, B_1) &= \Dir (u^+, B_1) + \Dir (u^-, B_1)\, ,\\
\Dir (u|_{\partial B_1}, \partial B_1) &= \Dir (u^+|_{\partial B_1}, \partial B_1) + \Dir (u^-|_{\partial B_1}, \partial B_1)\, .
\end{align*}
Assume now that the proposition is false and find a sequence of $\Dir$-minimizers $\{u_k\}\subset W^{1,2} (B_1, \IQSn)$ such that
\[
\Dir (u_k, B_1) \geq \frac{1}{m-2+(k+1)^{-1}} \Dir (u_k|_{\partial B_1}, \partial B_1)\, .
\]
After normalizing the maps we can assume that 
\[
\Dir (u_k, B_1) =1\, .
\]
We consider further the numbers
\[
\beta_k := \min \{|\{|u_k^+| = 0\}|, |\{|u_k^-|=0\}|\}\, 
\]
and, up to subsequences, we distinguish two cases: $\liminf_{k\to\infty} \beta_k >0$ and $\lim_{k\to\infty} \beta_k =0$.

\medskip

{\bf{First case.}} In this case we have the existence of a constant $\beta > 0$ such that
$|\{|u_k^+| = 0\}|\geq \beta$ and $|\{|u_k^-|=0\}|\geq \beta$ for every $k$. Since $|D|v||\leq |Dv|$ for any $Q$-valued map, we conclude from a classical variant of the Poincar\'e inequality that $\sup_k (\||u_k^+|\|_{L^2} + \||u_k^-|\|_{L^2}) < \infty$. Up to subsequences we can then assume that $u_k$ converges weakly in $W^{1,2}$ to some $u$ and the Proposition \ref{prop.compactness of minimizers} would imply that:
\begin{itemize}
\item $u$ is $\Dir$-minimizing;
\item $\Dir (u, B_1) = \lim_{k\to\infty} \Dir (u_k, B_1) =1$.
\end{itemize}
On the other hand, the semicontinuity of the Dirichlet energy would also imply that
\[
\Dir (u|_{\partial B_1}, \partial B_1) \leq \liminf_{k\to\infty} \Dir (u_k|_{\partial B_1}, \partial B_1) \leq m-2\, .
\]
So, according to Lemma \ref{l:0-lemma}, $u$ would have to be constant, which clearly is in contradiction with $\Dir (u, B_1) =1$.

\medskip

{\bf{Second case.}} In this case, again up to extraction of a subsequence, we can assume that $\lim_{k\to \infty} |\{|u_k^+|=0\}| = 0$. In turn
this implies that $\lim_{k\to \infty} |\{|u_k^-| > 0\}|=0$. In particular, since $\|D |u_k^-|\|_{L^2 (B_1)} \leq 1$, we get
$\lim_{k\to \infty} \||u_k^-|\|_{L^2 (B_1)} =0$. In turn this implies as well that $\left.|u_k^-|\right|_{\partial B_1}$ is bounded in $H^{1/2}$ and converges weakly to $0$ distributionally. Thus
\[
\lim_{k\to \infty} \int_{\partial B_1} |u_k^-|^2 = 0\, .
\]
Consider now the map $w_k : \partial B_1 \ni x \to (u_k^+ (x), 1)\in \Iqspec$ (where we have ``eliminated the negative part'' of $u_k$) and observe that
\[
\lim_{k\to \infty} \int_{\partial B_1} \mathcal{G}_s (w_k, u_k)^2 = 0\, .
\]
In particular, for $\lambda>0$ small (to be chosen later) use the Luckhaus Lemma \ref{lem.easyLuckhaus} to construct a function $h_k : B_1 \setminus B_{1-\lambda} \to \Iqspec$ with the properties that
\begin{itemize}
\item $h_k|_{\partial B_1} = u_k|_{\partial B_1}$;
\item $h_k (x) = w_k ((1-\lambda)^{-1} x)$ for every $x\in \partial B_{1-\lambda}$;
\item The following estimate holds with a constant $C$ independent of $\lambda$:
\[
\limsup_{k\to \infty} \Dir (h_k, B_1\setminus B_{1-\lambda}) \leq C \lambda\, .
\]
\end{itemize}
Now, we use \cite[Proposition 3.10]{DLS_Qvfr} to find a map $z_k\in W^{1,2} (B_{1-\lambda}, \Iqs)$ with the property that
\[
\Dir (z_k, B_{1-\lambda}) \leq \frac{1-\lambda}{m-2+\gamma} \Dir (u_k^+|_{\partial B_{1-\lambda}}, \partial B_{1-\lambda}) \leq \frac{(1-\lambda)^{m-2}(m-2+(k+1)^{-1})}{m-2+\gamma}\, ,
\]
where $\gamma = \gamma (m,n,Q)>0$. Clearly the map
\[
\hat{u}_k (x) = 
\begin{cases}
h_k (x) &\mbox{if $1\geq |x|\geq 1-\lambda$}\\
(z_k (x), 1)\in \Iqspec \qquad &\mbox{if $|x|\leq 1-\lambda$}
\end{cases}
\]
is in $W^{1,2} (B_1, \Iqspec)$ and has the same trace as $u_k$ on $\partial B_1$. By minimality 
\begin{align*}
1 &= \lim_{k\to \infty} \Dir (u_k, B_1) \leq \limsup_k \Dir (\hat{u}_k, B_1) \leq \frac{(1-\lambda)^{m-2}(m-2)}{m-2+\gamma} + C\lambda\, .
\end{align*}
Observe that $\lambda$ can be chosen arbitrarily small. On the other hand,
\[
\lim_{\lambda \to 0} \left(\frac{(1-\lambda)^{m-2}(m-2)}{m-2+\gamma} + C\lambda\right) = \frac{m-2}{m-2+\gamma} <1\, ,
\] 
which gives a contradiction, thus completing the proof of Proposition \ref{prop.energy boundary control}, and, in turn, of Theorem \ref{thm.interior Hoelder regularity}.

\section{Monotonicity of the frequency function}

As in the case of ``classical'' $Q$-valued functions, we introduce a suitable frequency function for maps taking values in $\Iqspec$.

\begin{definition}\label{d:fequency}
Let $\Omega\subset \R^m$ be an open set and consider a map $u\in W^{1,2} (\Omega, \Iqspec)$. For every $x\in \Omega$ and every $r\in ]0, \dist (x, \partial \Omega)[$ we define 
\begin{align*}
D_{x,u} (r) &:= \Dir (u, B_r (x))\\
H_{x,u} (r) &:= \int_{\partial B_r (x)} \G(u,Q\a{0})^2\, .
\end{align*}
Moreover, if $H_{x,u} (r) >0$, we define the frequency function
\[
I_{x,u} (r) :=\frac{r D_{x,u} (r)}{H_{x,u} (r)}\, .
\]
\end{definition}

If the point $x$ and the function $u$ are clear from the context, we will drop the subscripts from the corresponding quantities. In our context the celebrated monotonicity theorem of Almgren for the frequency function remains valid. 
More precisely we have the following theorem.

\begin{theorem}\label{t:frequency_monot}
Let $\Omega$ be a bounded open set and $u\in W^{1,2} (\Omega, \Iqspec)$ a $\Dir$-minimizing map. Fix a point $x\in \Omega$ and let $\rho:= \dist (x, \partial \Omega)$.
Then either $u^+ \equiv u^- \equiv Q \a{0}$ on $B_\rho (x)$ or 
$H_{x,u} (r) >0$ for every $r\in ]0, \rho[$ and in particular $I_{x,u} (r)$ is well defined.
Moreover, in the latter case:
\begin{itemize}
\item[(a)] The function $r\mapsto I_{x,u} (r)$ is monotone nondecreasing and therefore
\[
I_0 := \lim_{r\to 0} I_{x,u} (r)
\]
exists and is finite.
\item[(b)] $I_0 =0$ if and only if $\max \{ |u^+ (x)|, |u^- (x)|\} >0$.
\item[(c)] There is a positive constant $c_0 (m,n,Q)$ such that, if $u(x) = Q\a{0}$, then
$I_0 \geq c_0 (m,n,Q)$.
\item[(d)] The function $r\mapsto I_{x,u} (r)$ is constant if and only if $u|_{B_r (x)}$ is $I_0$-homogeneous, i.e. for each $y\in \mathbb S^{m-1}$ one of the following alternatives holds:
\begin{align}
&u (ry) = (u^+ (ry), 1) \qquad\mbox{and}\qquad u^+ (ry) = \sum_i \a{r^{I_0} u_i^+ (y)}\quad
&\forall r\in ]0, \rho[\, ,\\
&u (ry) = (u^- (ry), -1) \qquad \mbox{and}\qquad u^- (ry) = \sum_i \a{r^{I_0} u_i^- (y)}
&\forall r\in ]0, \rho[\, ,\\
&u(ry) = \left(Q \a{r^{I_0}\etab \circ u^+ (y)},1\right) =
\left( Q\a{r^{I_0} \etab \circ u^- (y)}, -1 \right) &\forall r\in ]0, \rho[\, .
\end{align}
\end{itemize}
\end{theorem}

The theorem follows from some important identities which we summarize in the following 
proposition.

\begin{proposition}\label{p:frequency}
Let $\Omega, x,\rho$ and $u$ be as in Theorem \ref{t:frequency_monot}. Then the maps
$r\mapsto H(r), D(r)$ are both absolutely continuous and the following identities hold for a.e.
$r\in ]0,\rho[$:
\begin{align}
D' (r) &= \frac{m-2}{r} D(r) +2 \int_{\partial B_r (x)} |\partial_\nu u|^2\label{e:D'}\\
H' (r) &= \frac{m-1}{r} H(r) + 2 D(r)\, .\label{e:H'}
\end{align}
Moreover there is a constant $C_0 (m,n,Q)$ such that, if $u(x) = Q \a{0}$, then
\begin{equation}\label{e:lower_bound_I}
H (r) \leq C_0 r D(r) \qquad \forall r \in ]0, \rho[\, .
\end{equation}
\end{proposition}
\begin{proof} Without loss of generality we can assume $x=0$.
The absolute continuity of the map $r\mapsto D(r)$ is an obvious consequence of the absolute continuity of integrals. Passing in polar coordinates we easily see that
\[
D' (r) = \int_{\partial B_r} |Du|^2
\]
for a.e. $r\in ]0, \rho[$. The identity \eqref{e:D'} is then an obvious consequence of \eqref{e:inner2}. 

Next consider a classical Sobolev $f$ and let us write 
\[
\int_{\partial B_r} f^2 = r^{m-1} \int_{\partial B_1} f^2 (rx)\, dx
\]
Differentiating in $r$ we get the distributional identity
\[
\frac{d}{dr} \int_{\partial B_r} f^2 = \frac{m-1}{r} \int_{\partial B_r} f^2 + \int_{\partial B_r} \langle \nabla f^2 (x), r^{-1} x \rangle\, dx\, ,
\]
which easily shows the absolute continuity of the function. We apply the latter identity with $f=|u|$ and use the chain rule formulas analogous to \cite[Section 1.3.1]{DLS_Qvfr}
to then derive 
\[
H' (r) = \frac{m-1}{r} H (r) + 2 \int_{\partial B_r} \sum_i \langle \partial_\nu u_i, u_i\rangle \, . 
\]
The identity \eqref{e:H'} is then a consequence of \eqref{e:outer}.

Finally, in order to show \eqref{e:lower_bound_I} observe first that we can assume, without loss of
generality, $r=1$. We then use the interior H\"older regularity Theorem \ref{thm.interior Hoelder regularity} to derive 
\[
\int_{\partial B_1} |u (sx)|^2\, dx \leq C \Dir (u, B_1)\, \text{ for all } s\in \left[0, \frac12\right]\,.
\]
Next differentiating the function  $s\mapsto \int_{\partial B_1} |u(sx)|^2\,dx$ and integrating in $s\in \left[\frac12,1\right]$ we easily conclude
\[
M := \max_{s\in [1/2,1]} \int_{\partial B_1} |u (sx)|^2 \leq C \int_{1/2}^1 \int_{\partial B_s} |u||Du|\, ds + C\, \Dir (u, B_1)\, . 
\]
In particular we derive 
\[
M \leq C M^{\frac{1}{2}} (\Dir (u, B_1))^{\frac{1}{2}} + C \Dir (u, B_1) \leq \frac{M}{2} + C\, \Dir (u, B_1)\, ,
\]
from which $M\leq C \Dir (u, B_1) = D(1)$ easily follows. Since $H(1) \leq M$, this gives the desired bound. 
\end{proof}

\begin{proof}[Proof of Theorem \ref{t:frequency_monot}]
Without loss of generality we can assume $x=0$ and $\rho=1$.
First of all, if $H(1) =0$, then clearly the map identically equal to $Q  \a{0}$ in $B_1$ is a competitor, hence by minimality it has to be $D(1) =0$ and thus $u\equiv Q\a{0}$ in $B_1$.  Hence, let us consider the case
$H(1) >0$. By Proposition \ref{p:frequency}, $H$ will be positive in a neighborhood of $1$ and thus we can consider the smallest $r_0<1$ for which $H > 0$ on $]r_0, 1[$. On such interval we can differentiate in $r$ and, using the identities \eqref{e:D'}, \eqref{e:H'} and \eqref{e:outer} compute
\begin{equation}\label{e:I'>0}
I' (r) = \frac{2\,r}{H(r)^2} \left[ \int_{\partial B_r} |\partial_\nu u|^2 \int_{\partial B_r} |u|^2
- \left(\int_{\partial B_r} \sum_i \langle \partial_\nu u_i, u_i\rangle \right)^2\right]\geq 0\, .
\end{equation}
In particular we conclude that $I$ is monotone on $]r_0,1[$ and so
\[
H(r) \geq \frac{r}{I (1)} D (r)\, . 
\]
Now, if it were $r_0 > 0$, then we would have $H (r_0) =0$ and, by \eqref{e:H'}
\[
H' (r) \leq \frac{m-1 + I (1)}{r} H (r) \qquad \mbox{for a.e. $r\in ]r_0, 1[$.}
\]
But then the usual Gronwall's lemma would imply that $H$ vanishes on $]r_0, 1[$, which is a contradiction.

We have thus proved the first claim of the theorem, namely that $H>0$ in $]0,1[$ under the assumption that $u$ is nontrivial in $B_1$. Moreover \eqref{e:I'>0} shows (a). (c) is now an obvious consequence of \eqref{e:lower_bound_I}, which in turn shows that $I_0=0$ implies $u(0) \neq Q\a{0}$. Now, if $u(0)\neq Q\a{0}$, namely $|u (0)|>0$, then by Theorem \ref{thm.interior Hoelder regularity} we have that 
\[
\lim_{r\to 0} \frac{H(r)}{r^{m-1}} = \Ha^{m-1}(\partial B_1) \, |u (0)|^2 >0\, .
\]
On the other hand by Theorem \ref{thm.interior Hoelder regularity} we have $D(r) \leq r^{m-2+2\alpha_0} D(1)$. Combining these two facts we then discover that $\lim_{r\to 0} I(r) =0$.

We finally come to (d). If $u$ is $I_0$-homogeneous, then the usual chain rules imply 
that $\partial_\nu u_i (x) = I_0 |x|^{-1} u_i (x)$ for a.e. $x$ and so we conclude that $I'$ is identically $0$. On the other hand if  $I' \equiv 0$, recalling that $H (r) >0$, we conclude the existence of a function $\lambda (r)$ such that then $\partial_\nu u_i (x) =\lambda (r) u_i (x)$ holds for a.e. $r$ and a.e. $x\in \partial B_r$.  On the other hand, this would imply
\[
I_0 = I (r) \stackrel{\eqref{e:outer}}{=} \frac{r\int_{\partial B_r} \sum_i \langle \partial_\nu u_i, u_i \rangle}{H(r)} = r \lambda (r)\, . 
\]
Hence we have 
\[
\partial_\nu u_i (x) =\frac{I_0}{|x|} u_i (x) \qquad \mbox{for a.e. $x\in B_1$.}  
\]
In particular the same identity holds for $u^+$ a.e. on $B_1^+$, for $u^-$ a.e. on $B_1^-$ and for the classical function $\etab \circ u$ everywhere. Since however $u^+= Q\a{\etab\circ u }$ on  $B_1\setminus B_1^+$, and $u^- (x) = Q \a{\etab\circ u}$ on $B_1\setminus B_1^-$, we conclude that the same identity actually holds a.e. on $B_1$ for all the three functions $u^+, u^-$ and $\etab \circ u$. In particular, by the argument given in \cite[Section 3.4.2]{DLS_Qvfr}, we conclude that all of them are $I_0$-homogeneous. This in turn implies (d) and completes the proof. 
\end{proof}

\section{Blow-up and estimate of the singular set}

\begin{definition}\label{d:regular_singular}
Given a $\Dir$-minimizer $u\in W^{1,2} (\Omega, \Iqspec)$, we say that a point $x\in \Omega$ is regular if there is a neighborhood $U$ of $x$ such that 
\begin{itemize}
\item[(a)] $u$ coincides with $(u^+, 1)$ in $U$ and $x$ is a regular point for the $\Dir$-minimizer $u^+\in W^{1,2} (U, \Iqs)$;
\item[(b)] Or $u$ coincides with $(u^-, -1)$ in $U$ and $x$ is a regular point for the $\Dir$-minimizer $u^-\in W^{1,2} (U, \Iqs)$.
\end{itemize}
The set of regular points will be denoted by $\reg (u)$, whereas its complement, the set of singular
points, will be denoted by $\sing (u)$. 
\end{definition}
Note that (a) and (b) are not mutually exclusive: they can both hold, in which case both $u^+$ and $u^-$ coincide with $Q \a{\etab\circ u}$ in $U$. 

\begin{theorem}\label{t:est_sing_Dir-min}
Let $\Omega\subset \mathbb R^m$ be a bounded open set, and let $u\in W^{1,2} (\Omega, \Iqspec)$ be a $\Dir$-minimizer. Then the Hausdorff dimension of $\sing (u)$ is at most $m-1$.
\end{theorem}

First of all observe that, by continuity, both $\Omega_+$ and $\Omega_-$ are open sets. Moreover, in the respective sets $u^+$ and $u^-$ are minimizers taking values in $\Iqs$. Since $\sing (u)\cap \Omega^\pm = \sing (u^\pm|_{\Omega^\pm})$, we easily conclude from
\cite[Theorem 0.11]{DLS_Qvfr} that the dimension of $\sing (u) \cap (\Omega_+\cup \Omega_-)$
is at most $m-2$. It remains to study $\sing (u) \cap \Omega_0$. On the other hand since $u = Q\a{\etab \circ u}$ on $\Omega_0$ it follows immediately that $\reg (u) \cap \Omega_0$ consists of the interior of $\Omega_0$. Thus the theorem will follow immediately from the following

\begin{proposition}\label{p:est_sing_Dir-min}
Consider a connected bounded open set $\Omega\subset \mathbb R^m$ and let $u \in W^{1,2}(\Omega, \Iqspec)$ be a $\Dir$-minimizer. If the dimension of $\Omega_0$ is strictly larger than $m-1$, then $\Omega_0=\Omega$.
\end{proposition}
\begin{proof}
The proof is entirely analogous to the one of \cite[Proposition 3.22]{DLS_Qvfr} and we just sketch it here for the reader's convenience. First of all we observe that without loss of generality we can assume $\etab\circ u \equiv 0$. In this case the statement of the proposition becomes then that either $\Dir (u)=0$, or the Hausdorff dimension of $\Omega_0$ is at most $m-1$. We
argue by contradiction and assume that $ \Dir (u)>0$ and $\mathcal{H}^{m-1+\alpha}_\infty (\Omega_0) > 0$ for some $\alpha > 0$. 

We then fix a point $x \in \Omega$ where 
\begin{equation}\label{e:density_point}
\limsup_{r\downarrow 0} \frac{\mathcal{H}^{m-1+\alpha}_\infty (\Omega_0 \cap B_r (x))}{r^{m-1+\alpha}}>0\, ,
\end{equation}
which for measure theoretic reasons occurs at $\mathcal{H}^{m-1+\alpha}$-a.e. $x\in \Omega_0$. 
For any $x\in \Omega$ we define $\rho (x) := \dist (x, \partial \Omega)$. We then claim that
for at least one $x$ where \eqref{e:density_point} holds we must have $\Dir (u, B_{\rho (x)} (x)) > 0$. Otherwise $\Dir (u, B_{\rho (x)} (x)) =0$ for every $x\in \Omega_0$ by a simple density argument. This would imply that $\Omega_0$ is open. Since it is also obviously closed and $\Omega$ is connected, we conclude that $\Omega_0 = \Omega$, which is a contradiction to $\Dir (u, \Omega) >0$.

Fix then a point $x$ where \eqref{e:density_point} holds and $\Dir (u, B_{\rho (x)} (x)) >0$. We take advantage of Theorem \ref{t:frequency_monot} in order to consider $I_0 = \lim_{r\downarrow 0} I_{x, u} (r)$, and we define the rescaled functions
\[
y \mapsto u_r (y)\, ,
\] 
where $u^\pm_r (y) = \sum_i \a{r^{-I_0} u_i^\pm (ry +x)}$. Using the compactness of $\Dir$-minimizers and the monotonicity of the frequency function, we conclude that, up to subsequences, rescaled maps converge (locally strongly in $W^{1,2}(\R^m,\Iqspec)$) to tangent functions defined on $\R^m$ which are locally $\Dir$-minimizers, take values in $\Iqspec$ and are nontrivial. In turn,
for an appropriate chosen subsequence, \eqref{e:density_point} is used with Theorem \ref{thm.interior Hoelder regularity} and with the upper semicontinuity of the $\mathcal{H}^{m-1+\alpha}_\infty$ measure to conclude that at least one such tangent function $v$ has the property that $\mathcal{H}^{m-1+\alpha} (\{|v|=0\}\cap B_1)> 0$.

Observe that $v$ is $I_0$ homogeneous. We can repeat the procedure and find a tangent function to $v$ at some $y$ with all the properties above. Such function turns out to be independent of the variable $y$. Repeating the construction $m$ times we end up with a function $w$ which has positive Dirichlet energy, is a local energy minimizer, is constant and for which the set $\{|w|=0\}$ is nonempty. This is clearly a contradiction.
\end{proof}

\section{Currents associated to normal graphs on an oriented submanifold}

The remaining sections of this work are aimed at obtaining several additional results concerning the geometry of (the integer rectifiable currents associated to) graphs of $\Iqspec$-valued functions, which will play a pivotal role in the approximation procedure of area minimizing currents modulo $p=2Q$ at points of density $Q$ carried out in \cite{DLHMS}.

From now on, we will often work under the following assumptions.

\begin{ipotesi}\label{a:Iqspec-normal}
We consider:

\begin{itemize}

\item[(M)] an open submanifold $\mathcal{M} \subset \R^{m+n}$ of class $C^3$ and dimension $m$, with $\Ha^m(\mathcal{M}) < \infty$, which is the graph of a function $\bphi \colon \Omega \subset \R^m \to \R^n$ with $\| \bphi \|_{C^3} \leq c$. We will let $A$ and $H$ denote the second fundamental form and the mean curvature vector of $\mathcal{M}$ as a submanifold of $\R^{m+n}$, respectively;

\item[(U)] a regular tubular neighborhood ${\bf U}$ of $\mathcal{M}$ in $\R^{m+n}$, defined as
\[
{\bf U} = \left\lbrace x + {\rm v} \, \colon \, x \in \mathcal{M}\,, {\rm v} \in T_{x}^{\perp}\mathcal{M} \mbox{ with } \abs{{\rm v}} \leq c_0 \right\rbrace\,,
\]  
where the constant $c_0$ is so small that a unique nearest point projection $\p \colon {\bf U} \to \mathcal{M}$ is well defined and of class $C^2$;

\item[(N)] a proper Lipschitz map $N \colon \mathcal{M} \to \mathscr{A}_Q(\R^{m+n})$ which satisfies $N^+_i (x), N^-_i (x), \etab\circ N (x) \in T_x^\perp \cM$ $\forall i$ and $\forall x\in \cM$; the map $N$ induces an
\[
F \colon \mathcal{M} \to \mathscr{A}_Q(\R^{m+n})
\]
by setting
\begin{align*}
F(x) &= \left\{
\begin{array}{ll}
\left(\sum_i \a{x + N^-_i (x)}, -1 \right) \qquad &\mbox{on $\cM_-$}\\
\left(\sum_i \a{x+N^+_i (x)}, +1 \right) \qquad &\mbox{on $\cM_+\cup \cM_0$}\, . 
\end{array}\right.
\end{align*}
\end{itemize}
\end{ipotesi}

Observe that $F^\pm$ and $\etab \circ F$ are proper maps, and they are Lipschitz continuous by Corollary \ref{c:separate_Lip}. Let $\a{\cM}$ be the multiplicity one $m$-dimensional current associated to $\cM$ with the orientation induced by its graph structure. Then, recalling \cite{DLS_Currents}, we have a natural way of pushing forward $\a{\cM}$ through the multivalued map $F^\pm$: the corresponding notation
is $\bT_{F^\pm}$ (in order to distinguish it from the classical ``push-forward'' via one-valued functions). 

\begin{definition}\label{d:push-mod-p} We introduce the notation $\bT_F$ for the integer rectifiable current which is naturally induced by $F$ and which is a representative $\modp$.
More precisely, we set
\begin{equation} \label{eq:push-mod-p}
\bT_F = \bT_{F^+} \res \p^{-1} (\cM_+) - \bT_{F^-} \res \p^{-1} (\cM_-) + Q (\etab\circ F)_\sharp \a{\cM_0}\, 
\end{equation}
and we introduce the notation
\begin{align}
\bT_F^+ & :=  \bT_{F^+} \res \p^{-1} (\cM_+)\\
\bT_F^- &:= - \bT_{F^-} \res \p^{-1} (\cM_-) \\
\bT_F^0 &:= Q (\etab\circ F)_\sharp \a{\cM_0}\, .
\end{align}
\end{definition}

\begin{remark}
Observe that $\|\bT_F^0\| (\bU \setminus \p^{-1} (\cM_0)) =0$. In particular,
since the sets $\p^{-1}(\cM_+), \p^{-1} (\cM_-)$ and $\p^{-1} (\cM_0)$ are pairwise disjoint, for every Borel $E\subset \bU$ we have
\begin{equation}\label{e:mass_separates}
\|\bT_F\| (E) = \|\bT^0_F\| (E) + \|\bT^-_F\| (E) + \|\bT^+_F\| (E)\, .
\end{equation}
\end{remark}

\section{Compatible triples}

Suppose $(g^+,g^-,g)$ is a triple of Lipschitz continuous functions with $g^{\pm}: U \to \Iqs$ and $g: U \to \R^n$ with the additional property that they satisfy the following

\begin{definition}[Compatibility conditions]\label{d:compatibility}
$\,$
\begin{itemize}
	\item[(a)] For any $x \in U$, either ${\rm sep}(g^+(x)) = 0$ or ${\rm sep}(g^-(x)) = 0$.
	\item[(b)] $g(x) = \bfeta \circ g^+ (x)$ whenever ${\rm sep} (g^+ (x)) =0$ and $
	g(x) = \bfeta \circ g^-(x)$ whenever ${\rm sep} (g^- (x)) =0$.
\end{itemize}
Here, we have used the notation introduced in \cite{DLS_Qvfr}, according to which $\operatorname{sep}(T)= \min\{ \abs{t_i-t_j} \colon t_i\neq t_j, T= \sum_{l=1}^Q \a{t_l} \}$ with the convention that $\min\emptyset = 0$. Also note that if $\left( g^+, g^-, g\right)$ satisfies the compatibility conditions, and if for some $x \in U$ it holds ${\rm sep}(g^+(x)) = 0 = {\rm sep}(g^-(x))$ then necessarily $g^+(x) = g^-(x) = Q \a{v}$ for some $v \in \R^n$, and $g(x) = v$.
\end{definition}

To such a triple we can associate a Lipschitz map into $\Iqspec$ by means of the following transformation. We first define 
\[
{\bf j} \colon  (g^+, g^-, g) \mapsto (v, w, z) = {\bf j}(g^+,g^-,g) := (g^+ \ominus \etab \circ g^+,g^- \ominus \etab \circ g^- , \etab\circ g^+ + \etab\circ g^- - g). \]
Then, we map ${\bf j}(g^+, g^-, g)$ into $f := \iso^{-1} ({\bf j} (g^+, g^-, g))$. In particular, the $\Iqspec$-valued map $f$ can be explicitly given as 
\[
f (x) = \left\{
\begin{array}{ll}
(g^+ (x), 1) \quad &\mbox{if ${\rm sep} (g^- (x)) = 0$}\\
(g^- (x), -1) \quad &\mbox{otherwise.}
\end{array}\right.
\]
Consistently with the notation of the previous sections, since $\iso$ is an isometry, we identify $f$ and $(v,w,z) = {\bf j} (g^+, g^-, g)$ and use interchangeably both symbols, depending on which is most convenient at the moment. 
One readily checks that $f$ is a Lipschitz map from $U$ into $\Iqspec$. 

Note that moreover 
${\bf j} (f^+,f^-,\etab\circ f) = f$. We thus have a right inverse of the map ${\bf j}$. 
However, there is {\em not} a 1-to-1 correspondence between $\Iqspec$-valued maps $f$ and triples $(g^+, g^-, g)$ satisfying (a) and (b). We therefore introduce the following terminology.

\begin{definition}\label{d:canonical}
The triple $(f^+, f^-, \etab\circ f)$ will be called the {\em canonical decomposition} of the map $f$. 
\end{definition}
Next note that the following lemma is a very simple consequence of the above definitions.

\begin{lemma}\label{l:admissible_triples}
For any $f:U \to \Iqspec$ the following holds. 
\begin{itemize}
\item[(i)] First of all we have the estimates
\begin{align}
\max \{\Lip (f^\pm), \Lip (\etab\circ f)\} &\leq \Lip (f)\, ;\label{e:lipf_lipf+-}
\end{align}
and
\begin{equation}
\Lip (f) \leq \Lip (g^+) + \Lip (g^-)\, .\label{e:lipf+-_lipf}
\end{equation}
for any $(g^+, g^-, g)$ such that ${\bf j} (g^+, g^-, g) = f$.
\item[(ii)] The canonical decomposition of the domain $U$ of $f$ can be determined using any triple $(g^+, g^-, g)$ such that ${\bf j} (g^+, g^-, g) = f$. More precisely:
\begin{align}
U_\pm &= \{{\rm sep} (g^\pm)>0\}\label{e:U+-}\\
U_0 &= U\setminus (U_+ \cup U_-) = \{{\rm sep} (g^+) = {\rm sep} (g^-) =0\}\, .
\label{e:U0}
\end{align}
\item[(iii)] The following identities hold whenever ${\bf j} (g^+, g^-, g) = f$:
\begin{align}
\sum_i \a{Dg^+_i} &=\sum_i \a{Df_i} \qquad \mbox{a.e. on $U_+\cup U_0$}\, ,\label{e:comp_der_1}\\
\sum_i \a{Dg^-_i} &= \sum_i \a{Df_i} \qquad \mbox{a.e. on $U_-\cup U_0$}\, ,\label{e:comp_der_2}
\end{align}
\begin{equation}\label{e:comp_der_3}
|Df| =
\left\{
\begin{array}{ll}
|Dg^+| \qquad &\mbox{a.e. on $U_+$}\\
|Dg^-| \qquad &\mbox{a.e. on $U_-$}\\
\sqrt{Q} |D g| & \mbox{a.e. on $U_0$}.
\end{array}
\right.
\end{equation}
\end{itemize}
\end{lemma}

\section{Taylor expansion of area and excess}

We start with a series of theorems which are focused on Taylor expansions of the mass of $\bT_F$ and important variants. The first one, which corresponds to \cite[Theorem 3.2]{DLS_Currents} is the following. 

\begin{theorem}[Expansion of $\mass (\bT_F)$] \label{t:taylor_area}
If $\cM$, $N$ and $F$ are as in Assumption \ref{a:Iqspec-normal} and $\bar{c}$
is smaller than a geometric constant, then
\begin{align}
\mass (\mathbf{T}_F) = {}&  Q \, \cH^m (\cM) - Q\int_\cM \langle H, \etab\circ N\rangle
+ \frac{1}{2} \int_\cM |D N|^2\nonumber\\
&+ \int_\cM \sum_i \Big(P_2 (x,N_i) +  P_3 (x, N_i, DN_i) + R_4 (x, DN_i)\Big),  \label{e:taylor_area}
\end{align}
where
$P_2$, $P_3$ and $R_4$ are $C^1$ functions with the following properties:
\begin{itemize}
\item[(i)] $v\mapsto P_2(x, v)$ is a quadratic form on the normal bundle of $\cM$
satisfying
\begin{equation}\label{e:order_2}
|P_2 (x,v)|\leq C |A (x)|^2 |v|^2 \qquad\quad \forall \; x\in \cM, \;\forall\; v\perp T_x\cM;
\end{equation}
\item[(ii)] $P_3 (x, v, D)= \sum_i L_i (x,v) Q_i (x, D)$, where
$v\mapsto L_i (x,v)$ are linear forms on the normal bundle of $\cM$ and $D\mapsto Q_i (x,D)$ are quadratic forms on the space of ${(m+n)\times(m+n)}$-matrices,
satisfying
\begin{align*}
&|L_i (x,v)|\leq C |A(x)||v|
&\forall x\in \cM, \, \forall v\perp T_x\cM ,\\
&|Q_i (x, D)| \leq C |D|^2 &\forall x\in \cM\, ,\forall D \in \R^{(m+n)\times(m+n)}\,  ;
\end{align*}

\item[(iii)] $|R_4 (x,D)| = |D|^3 L (x,D)$, for some function $L$ with $\Lip (L)\leq C$,
which satisfies $L(x,0)=0$ for every $x\in \cM$
and is independent of $x$ when $A\equiv 0$.
\end{itemize}
Moreover, for any Borel function $h: \R^{m+n}\to \R$, 
\begin{equation}\label{e:taylor_aggiuntivo} 
\left| \int h\, d\|\bT_F\| - \int_{\cM} \sum_i h \circ F_i\right|
\leq C \int_{\cM} \Big(\sum_i |A| |h \circ F_i||N_i| + \|h\|_\infty (|DN|^2 + |A|^2 |N|^2)\Big), 
\end{equation}
and, if $h (q) = g (\p (q))$ for some $g$, we have 
\begin{align}
\left| \int h\, d\|\bT_F\| - \int_\cM (Q- Q\langle H, \etab\circ N\rangle + \textstyle{\frac{1}{2}}|DN|^2) \, g\right|
&\leq C \int_\cM \big(|A|^2 |N|^2 + |DN|^4\big) |g|\, .\label{e:taylor_aggiuntivo_2}
\end{align}
\end{theorem}

\begin{proof}
Observe that the first part of the statement is a simple consequence of \eqref{e:taylor_aggiuntivo_2}. The latter one can be easily reduced to \cite[Theorem 3.2]{DLS_Currents} using \eqref{e:mass_separates}. Indeed, if we introduce $g^+ := g \mathbf{1}_{\cM_+}$, $g^- := g \mathbf{1}_{\cM_-}$ and  $g^0: = g \mathbf{1}_{\cM_0}$ and the corresponding $h^\square (p) = g^\square (\p (p))$, it suffices to prove \eqref{e:taylor_aggiuntivo_2} for each pair $(h^\square, g^\square)$. In such cases, however, \eqref{e:taylor_aggiuntivo_2} can be concluded from \cite[Theorem 3.2]{DLS_Currents} (more specifically \cite[Eq. (3.4)]{DLS_Currents}) applied to each $\bT_{F^\square}$.

As for \eqref{e:taylor_aggiuntivo} precisely the same argument reduces it to prove it for
$h^+ := h \mathbf{1}_{\p^{-1} (\cM_+)}$, $h^- := h \mathbf{1}_{\p^{-1} (\cM_-)}$ and 
$h^0 := h \mathbf{1}_{\p^{-1} (\cM_0)}$. As above, each such case can be inferred from \cite[Theorem 3.2]{DLS_Currents} (more specifically \cite[Eq.  (3.3)]{DLS_Currents}) applied to the corresponding $\bT_{F^\square}$. 
\end{proof}

An important corollary of the theorem above is the following. 

\begin{corollary}[Expansion of $\mass (\mathbf{G}_f)$] \label{c:taylor_area}
Assume $\Omega\subset \R^m$ is an open set with bounded measure and $f:\Omega
\to \Iqspec$ a Lipschitz map with $\Lip (f)\leq \bar{c}$. Then,
\begin{equation}\label{e:taylor_grafico}
\mass (\mathbf{G}_f) = Q |\Omega| + \frac{1}{2} \int_\Omega |Df|^2 + \int_\Omega \sum_i \bar{R}_4 (Df_i)\, ,
\end{equation}
where $\bar{R}_4\in C^1$ satisfies $|\bar{R}_4 (D)|= |D|^3\bar{L} (D)$ for $\bar{L}$ with $\Lip (\bar{L})\leq C$ and
$\bar L(0) = 0$.
\end{corollary}

\begin{proof}
The statement follows from Theorem \ref{t:taylor_area} applied to the case in which $\mathcal{M}$ is flat: since $A=0$ (and thus $H=0$), the linear and third order terms in the expansion \eqref{e:taylor_area} vanish. 
\end{proof}

We next come to two further Taylor expansions. 

\begin{proposition}[Expansion of a curvilinear excess]\label{p:eccesso_curvo}
There exists a dimensional constant $C>0$ such that, if
$\cM$, $F$ and $N$ are as in Assumption \ref{a:Iqspec-normal} with $\bar{c}$ small enough, then
\begin{align}
\left\vert \int |\vec{\bT}_F (x) - \vec{\cM} (\p (x))|_{no}^2 \, d\|\bT_F\| (x)
- \int_\cM |DN|^2 \right\vert \leq  C \int_{\cM} (|A|^2 |N|^2 + |DN|^4)
\label{e:ecurvo_sopra}\,, 
\end{align}
where $\vec{\bT}_F$ and $\vec{\cM}$ are the unit $m$-vectors orienting $\bT_F$ and
$T\cM$, respectively, and $|\cdot|_{no}$ is the non-oriented distance defined by
\begin{equation}\label{e:no_excess_0}
|\vec{\bT}_F- \vec{\cM} (\p (x))|_{no} := \min \{|\vec{\bT}_F-\vec{\cM} (\p (x))|, |\vec{\bT}_F+\vec{\cM} (\p (x))|\}
\end{equation}
\end{proposition}
\begin{proof}
Proceeding as in the argument leading to Theorem \ref{t:taylor_area}, we can reduce the statement to corresponding ones where $\bT_F$ is replaced by $\bT_{F^\square}$ and $\cM$ is replaced by $\cM_\pm$ or $\cM_0$, after observing that
\[
\abs{ \vec{\bT}_{F}(x) -  \vec{\cM} (\p (x))   }^2_{no}  = \abs{\vec{\bT}_{F^\square}(x) - \vec{\cM} (\p(x))}^2 \qquad \mbox{at $\|{\bT}_F\|-$a.e. $x$}\,.
\]

Each of these statements can then be concluded from \cite[Proposition 3.4]{DLS_Currents}: note indeed that, although \cite[Proposition 3.4]{DLS_Currents} is ``global'', a local version (where, given any Borel $E\subset \cM$, in the right hand side of \cite[Eq. (3.13)]{DLS_Currents} $\bT_F$ is substituted by $\bT_F \res \p^{-1} (E)$ and in the left hand side $\cM$ is substituted by $E$) follows directly from the proof given there. 

\end{proof}

The final Taylor expansion which we treat in this section is the one of a suitable cylindrical excess. In the following theorem, given a disc $B_s \subset \R^m$ we will let $\bC_s$ denote the cylinder $\bC_s := B_s \times \R^n \subset \R^{m} \times \R^{n} \simeq \R^{m+n}$. 

\begin{theorem}[Expansion of a cylindrical excess]\label{t:taylor_tilt}
There exist dimensional constants $C, c>0$ with the following property.
Let $f: \R^m \to \Iqspec$ be a Lipschitz map with ${\rm Lip}\, (f)\leq c$.
For any $0<s$, set $ L:= \fint_{B_s} D (\etab \circ f)$ and denote by $\vec\tau$
the $m$-dimensional simple unit vector orienting the graph of the linear map
$y\mapsto L \cdot y$. Then, we have
\begin{equation}\label{e:taylor_tilt}
\left| \int_{\bC_s} \left| \vec{\mathbf{G}}_f - \vec\tau\right|_{no}^2\, d \|\mathbf{G}_f\| - \int_{B_s} \cG_s (D f, Q \a{L})^2 \right| \leq C \int_{B_s} |Df|^4\, .
\end{equation}
\end{theorem}

\begin{proof}
Denote by $E$ the quantity
\[
E = \int_{\bC_s} \left| \vec{\mathbf{G}}_f - \vec\tau\right|_{no}^2\, d \|\mathbf{G}_f\|\,.
\]
Observe that, if we set $U:= B_s$ and introduce the triples $(f^+, f^-, \etab\circ f)$ and $(U_+, U_-, U_0)$, we easily conclude that
\begin{align*}
E =\; & \int_{U_+\times \R^n} |\vec{\mathbf{G}}_{f^+} - \vec\tau|^2 d \|\mathbf{G}_{f^+}\| 
+  \int_{U_-\times \R^n} |\vec{\mathbf{G}}_{f^-} - \vec\tau|^2 d \|\mathbf{G}_{f^-}\|\\
& + Q \int_{U_0\times \R^n} |\vec{\mathbf{G}}_{\etab\circ f} - \vec\tau|^2 d\|\mathbf{G}_{\etab\circ f}\|\, .
\end{align*}

We next can apply the same computations of the proof of \cite[Theorem 3.5]{DLS_Currents} to arrive at
\begin{align*}
E = \; & \int_{U_+} |Df^+|^2 +  Q\,|U_+|\,|L|^2 - 2 \int_{U_+} \sum_i Df^+_i :L  \\
&+ \int_{U_-} |Df^-|^2 +  Q\,|U_-|\,|L|^2 - 2 \int_{U_-} \sum_i Df^-_i :L\\
&+ Q \int_{U_0} |D \etab\circ f|^2 + Q \, |U_0|\, |L|^2 - 2Q \int_{U_0} D \etab\circ f : L\\
& + O \left(\int_{B_s} |Df|^4\right)\, .
\end{align*}
This easily gives
\begin{align*}
E =\; & \int_{U_+} \cG (Df^+, Q \a{L})^2 + \int_{U_-} \cG (Df^-, Q \a{L})^2\\
&+ Q \int_{U_0} |D\etab \circ f - L|^2 + O \left(\int_{B_s} |Df|^4\right)\, .
\end{align*}
Using \eqref{e:comp_der_1} we then conclude \eqref{e:taylor_tilt}. 
\end{proof}

\section{Taylor expansion of first variations}

In this section we consider Taylor expansions of the first variations. 

We begin with the expansion for the first variation of graphs. In the following theorem, $\Lip_c(\Omega \times \R^n, \R^d)$ denotes the space of functions $\zeta \in \Lip(\Omega \times \R^n, \R^d)$ for which there exists $\Omega' \subset \subset \Omega$ such that $f(x,y) = 0$ when $x \notin \Omega'$.

\begin{theorem}[Expansion of $\delta \mathbf{G}_f (\chi)$]\label{t:grafici}
Let $\Omega\subset \R^m$ be a bounded open set and $f:\Omega \to \Iqspec$ a map 
with $\Lip (f)\leq \bar{c}$.
Consider a function $\zeta\in \Lip_c (\Omega \times \R^n, \R^n)$ and the corresponding vector field
$\chi\in \Lip_c (\Omega\times \R^n, \R^{m+n})$ given by $\chi (x,y) = (0, \zeta (x,y))$.
Then,
\begin{equation}\label{e:variazione_media}
\left|\delta \mathbf{G}_f (\chi) - \int_\Omega\sum_i  \big(D_x \zeta (x, f_i) + D_y \zeta (x, f_i) \cdot Df_i\big) : Df_i \right| 
\leq C \int _\Omega |D\zeta| |Df|^3\, .
\end{equation}
\end{theorem}

The next two theorems deal with general
$\bT_F$ as in Assumption \ref{a:Iqspec-normal}. We restrict our attention to ``outer and inner variations''. Outer variations result
from deformations of the normal bundle of $\cM$ which are the identity on $\cM$ and 
map each fibre into itself, whereas inner variations result from composing
the map $F$ with isotopies of $\cM$.

\begin{theorem}[Expansion of outer variations]\label{t:outer}
Let $\cM$, $\mathbf{U}$, $\p$ and $F$ be as in Assumption \ref{a:Iqspec-normal}
with $\bar{c}$ sufficiently small. If $\varphi\in \Lip_c (\cM)$ and $X(q) := 
\varphi (\mathbf{p} (q)) (q-\mathbf{p}(q))$, then
\begin{align}
\delta \mathbf{T}_F (X) = \int_\cM \Big(\varphi \, |DN|^2 + 
\sum_i (N_i \otimes D \varphi) : DN_i\Big) - \underbrace{Q \int_\cM \varphi \langle H, \etab\circ N\rangle}_{{\rm Err}_1} + \sum_{i=2}^3{\rm Err}_i
\label{e:outer_exp} 
\end{align}
where
\begin{gather}
|{\rm Err}_2| \leq C \int_\cM |\varphi| |A|^2|N|^2\label{e:outer_resto_2}\\
|{\rm Err}_3|\leq C \int_\cM \Big(|\varphi| \big(|DN|^2 |N| |A| + |DN|^4\big) +
|D\varphi| \big(|DN|^3 |N| + |DN| |N|^2 |A|\big)\Big)\label{e:outer_resto_3}\, .
\end{gather}
\end{theorem}

Let $Y$ be a Lipschitz vector field on $T\cM$ with compact
support, and define $X$ on $\bU$ setting $X (q) = Y (\p (q))$. Let $\{\Psi_\eps\}_{\eps \in ]-\eta, \eta[}$ be
any isotopy with $\Psi_0 = {\rm id}$ and $\left.\frac{d}{d\eps}\right\vert_{\eps=0}
\Psi_\eps = Y$ and define the following
isotopy of $\mathbf{U}$: $
\Phi_\eps (q) = \Psi_\eps (\mathbf{p} (q)) + (q-\mathbf{p} (q))$.
Clearly $X = \left.\frac{d}{d\eps}\right|_{\eps =0} \Phi_\eps$. 

\begin{theorem}[Expansion of inner variations]\label{t:inner}
Let $\cM$, $\mathbf{U}$ and $F$ be as in Assumption \ref{a:Iqspec-normal} with $\bar{c}$ sufficiently small.
If $X$ is as above, then
\begin{align}
\delta \mathbf{T}_F (X) = \int_\cM \Big( \frac{|DN|^2}{2} {\rm div}_{\cM}\, Y -
\sum_i  D N_i : ( DN_i\cdot D_{\cM} Y)\Big)+ \sum_{i=1}^3{\rm Err}_i,\label{e:inner} 
\end{align}
where
\begin{gather}
{\rm Err}_1 = - Q \int_{ \cM}\big( \langle H, \etab \circ N\rangle\, {\rm div}_{\cM} Y + \langle D_Y H, \etab\circ N\rangle\big)\, ,\label{e:inner_resto_1}\allowdisplaybreaks\\
|{\rm Err}_2| \leq C \int_\cM |A|^2 \left(|DY| |N|^2  +|Y| |N|\, |DN|\right), \label{e:inner_resto_2}\allowdisplaybreaks\\
|{\rm Err}_3|\leq C \int_\cM \Big( |Y| |A| |DN|^2 \big(|N| + |DN|\big) + |DY| \big(|A|\,|N|^2 |DN| + |DN|^4\big)\Big)\label{e:inner_resto_3}\, .
\end{gather}
\end{theorem}

The three theorems can all be proved appealing to the computations in \cite[Section 4]{DLS_Currents}. 
First of all, by a standard approximation procedure we can assume that the test vector fields are in fact smooth.
Next consider the case of Theorem \ref{t:outer}. Using the triple $F^+, F^-$ and $\etab\circ F$, and taking into account the fact that the currents $\bT_{F^+} \res \p^{-1}(\cM_+)$, $\bT_{F^-} \res \p^{-1}(\cM_-)$, and $\bT_{\etab \circ F} \res \p^{-1}(\cM_0)$ are supported on disjoint sets, we can compute
\begin{align}
\delta \bT_F (X) = \delta \bT_{F^+}\res \p^{-1} (\cM_+) (X) + \delta \bT_{F^-}\res \p^{-1} (\cM_-) (X) + Q\, \delta \bT_{\etab \circ F}\res \p^{-1} (\cM_0) (X)\, .\label{e:decompo_first_var}
\end{align}
We can then appeal to \cite[Theorem 4.2]{DLS_Currents} to get the corresponding Taylor expansions of the three pieces separately and use \eqref{e:comp_der_1}, \eqref{e:comp_der_2} and \eqref{e:comp_der_3} to conclude the desired formulas. The proof of Theorem \ref{t:grafici} is entirely analogous, using \cite[Theorem 4.1]{DLS_Currents}. In both cases there is only one thing to notice: although in the statements of \cite[Theorem 4.1 \& Theorem 4.2]{DLS_Currents} the domain is assumed to be an open set (and the map $\varphi$ in  \cite[Theorem 4.2]{DLS_Currents} is assumed to have compact support), it can be easily seen that the proof given in \cite{DLS_Currents} is not using any specific property of the domain of the map except for its Borel measurability (and the assumption on the support of the map $\varphi$ in \cite[Theorem 4.2]{DLS_Currents} is also redundant).

Reducing Theorem \ref{t:inner} to  the case of \cite[Theorem 4.3]{DLS_Currents} is however different, since in the final part of the proof one integration by parts is used to treat the linear error term and thus the assumption that the domain is open and that the vector field $Y$ has compact support is crucial. In this case we proceed instead as follows. First of all we decompose the first variation of $\bT_F$ as in \eqref{e:decompo_first_var} and we denote by $N^+, N^-$ and $\etab\circ N$ the triple corresponding to the $\Iqspec$-valued map ``normal part'' $N$. For each of the three summands in \eqref{e:decompo_first_var} we then follow the proof of \cite[Theorem 4.3]{DLS_Currents} till \cite[Eq. (4.13)]{DLS_Currents}. Taking $\delta \bT_{F^+}$ as an example, we get the expansion
\begin{equation}\label{e:expansion_inner_+}
\begin{split}
 &\delta \bT_{F^+} \res \p^{-1} (\cM_+) (X)\\
=\; & \int_{\cM_+} \Big( \frac{|DN^+|^2}{2} {\rm div}_{\cM}\, Y -
\sum_i  D N^+_i : ( DN^+_i\cdot D_{\cM} Y)\Big) + J^+_2 + {\rm Err}^+_2 + {\rm Err}^+_3\, ,
\end{split}
\end{equation}
where
\begin{align}
J^+_2 &= Q \int_{\cM_+} \sum_j \left(
\langle A (e_j, \nabla_{e_j} Y), \etab\circ N^+\rangle + \langle A( e_j, Y), D_{e_j} \etab \circ N^+\rangle\right)\\
{\rm Err}^+_2 & \leq \int_{\cM_+} |A|^2 \left(|DY| |N^+|^2  +|Y| |N^+|\, |DN^+|\right), \label{e:inner_resto_4}\\
{\rm Err}^+_3 &\leq \int_\cM \Big( |Y| |A| |DN^+|^2 \big(|N^+| + |DN^+|\big) + |DY| \big(|A|\,|N^+|^2 |DN| + |DN^+|^4\big)\Big)\, .
\end{align}
We next sum to \eqref{e:expansion_inner_+} the corresponding expansions for the other two summands in the decomposition of $\delta \bT_F (X)$ (namely $\delta \bT_{F^-}
\res \p^{-1} (\cM_-) (X)$ and $Q\, \delta \bT_{\etab \circ F}\res \p^{-1} (\cM_0) (X)$). Using then \eqref{e:comp_der_1}, \eqref{e:comp_der_2} and \eqref{e:comp_der_3}, we easily reach
\begin{align}
\delta \bT_{F} (X) =\; & \int_{\cM} \Big( \frac{|DN|^2}{2} {\rm div}_{\cM}\, Y -
\sum_i  D N_i : ( DN_i\cdot D_{\cM} Y)\Big) + J_2 + {\rm Err}_2 + {\rm Err}_3\, ,
\end{align}
where ${\rm Err}_2$ and ${\rm Err}_3$ satisfy the estimates claimed in Theorem \ref{t:inner} and 
\begin{align}
J_2 = Q \int_\cM \sum_j \left(
\langle A (e_j, \nabla_{e_j} Y), \etab\circ N\rangle + \langle A( e_j, Y), D_{e_j} \etab \circ N\rangle\right)\, .
\end{align}
Note that at this stage the term $J_2$ corresponds to the term $J_2$ of \cite[Eq. (4.17)]{DLS_Currents}. Thus we can follow the remaining part of the proof of \cite[Theorem 4.3]{DLS_Currents} where an integration by parts transforms $J_2$ into the term ${\rm Err}_1$ of the expansion \eqref{e:inner}. 

\section{Reparametrization theorem on normal bundles}

In this section we state and prove the analogues of the results in \cite[Section 5]{DLS_Currents} in the context of $\Iqspec$-valued maps. 

\begin{theorem}[$\Iqspec$ parametrizations]\label{t:cambio}
Let $Q,m,n\in \N$ and $s<r<1$.
Then, there are constants $c_0, C>0$ (depending on $Q,m,n$ and $\frac{r}{s}$) with the following property.
Let $\phii$, $\cM$ and $\bU$ be as in Assumption~\ref{a:Iqspec-normal} with
$\Omega = B_{s}$ and let $f: B_{r} \to \Iqspec$ be such that
\begin{equation}\label{e:bounds_phi_f}
\|\phii\|_{C^2} + \Lip (f)\leq c_0 \qquad \mbox{and} \qquad \|\phii\|_{C^0} + \|f\|_{C^0} \leq c_0\, r .
\end{equation}
Set $\Phii (x) := (x, \phii (x))$.
Then, there are maps $F$ and $N$ as in Assumption~\ref{a:Iqspec-normal}(N) such that
$\bT_F = \mathbf{G}_f \res \bU$ and
\begin{gather}
\Lip (N) \leq C \big(\|D^2\phii\|_{C^0} \|N\|_{C^0} + \|D\phii\|_{C^0} +\Lip (f)\big)\, ,\label{e:stime_cambio2}\\
\frac{1}{2\sqrt{Q}} |N (\Phii (x))| \leq \cG_s (f (x), Q\a{\phii (x)})\leq 2 \sqrt{Q} \, |N (\Phii (x))|\qquad 
\forall x\,\in B_s\label{e:stime_cambio1}\, , \\
|\etab \circ N (\Phii (x))| \leq C |\etab\circ f (x) - \phii (x)| + C \Lip (f) |D\phii (x)| |N(\Phii (x))|\, 
\qquad \forall x\,\in B_s.\label{e:stima_media}
\end{gather} 
Finally, assume $x \in B_s$ and
$(x, \etab\circ f (x)) = \xi + {\rm v}$ for some $\xi\in \cM$ and ${\rm v}\perp T_\xi \cM$.
Then,
\begin{equation}\label{e:stima_cambio_10}
\cG_s (N(\xi), Q\a{{\rm v}}) \leq 2\sqrt{Q} \,\cG_s (f(x), Q\a{\etab\circ f(x)})\, . 
\end{equation}
\end{theorem}

For further reference, we state
the following immediate corollary of Theorem~\ref{t:cambio}, corresponding to the case
of a linear $\phii$. In the statement we shall adopt the following notation: if $\pi$ is an $m$-dimensional linear subspace (briefly, an $m$-plane) in $\R^{m+n}$, $x \in \R^{m+n}$, and $r>0$, then we set $B_r(x,\pi) := \mathbf{B}_r(x) \cap (x+\pi)$, where $\mathbf{B}_r(x)$ is the open ball centered at $x$ with radius $r$ in $\R^{m+n}$, and we will only write $B_r(\pi)$ if $x$ is the origin. Furthermore, we shall use the symbol $\mathscr{A}_Q(\pi)$ to denote the space of special $Q$-points in the plane $\pi$. 

\begin{proposition}[$Q$-valued graphical reparametrization]\label{p:cambio_lineare}
Let $Q, m, n\in \mathbb N$ and $s<r<1$. There exist positive constants $c,C$ (depending only on
$Q,m,n$ and $\frac{r}{s}$) with the following
property. Let $\pi_0$ and $\pi$ be $m$-planes
with $|\pi-\pi_0|\leq c$ and $f:B_{r} (\pi_0)\to \mathscr{A}_Q(\pi_0^\perp)$ with 
$\Lip (f) \leq c$ and $|f|\leq c r$. Then, there is a Lipschitz
map $g: B_s (\pi)\to \mathscr{A}_Q(\pi^\perp)$ with $\mathbf{G}_g = \mathbf{G}_f \res \bC_s (\pi)$ and
such that the following estimates hold on $B_s (\pi)$:
\begin{gather}
\|g\|_{C^0} \leq C r |\pi-\pi_0| + C\|f\|_{C^0},\\
\Lip (g)\leq C |\pi-\pi_0| + C \Lip (f)\, .
\end{gather}
\end{proposition}

Again Theorem \ref{t:cambio} will be reduced to the corresponding \cite[Theorem 5.1]{DLS_Currents}. 
First of all we introduce the triple $(f^+, f^-, \etab\circ f)$ and, setting $U := B_r$, consider $U_+, U_-$ and $U_0$. For each of the maps $f^+, f^-, \etab\circ f$ we apply \cite[Theorem 5.1]{DLS_Currents} and find the corresponding ``parametrizations'', which we denote $G^+, G^-, g$, so that 
\begin{align}
\bT_{G^+} &= \mathbf{G}_{f^+}\res \bU\\
\bT_{G^-} &= \mathbf{G}_{f^-}\res \bU\\
\bT_{g}  &= \mathbf{G}_{\etab\circ f}\res \bU\, .
\end{align}
We now wish to show two things, which we summarize in the following

\begin{lemma}\label{l:redistribution_lemma}
The triple $(G^+, G^-, g)$ satisfies the compatibility conditions of Definition \ref{d:compatibility}, and 
the map $F= {\bf j} (G^+, G^-, g)$ satisfies $\bT_F = \mathbf{G}_f \res \bU$. In fact the following stronger conclusion holds:
\begin{align}
\bT_{G^+} \res \p^{-1} (\cM_+) &= \bT_{F^+} \res \p^{-1} (\cM_+) = \mathbf{G}_{f^+} \res \bU \cap (U_+\times \mathbb R^n)\label{e:pos=pos}\\
\bT_{G^-} \res \p^{-1} (\cM_-) &= \bT_{F^-} \res \p^{-1} (\cM_-) = \mathbf{G}_{f^-}\res \bU \cap (U_- \times \mathbb R^n)\label{e:neg=neg}\\
\bT_{g} \res \p^{-1} (\cM_0)  &= \bT_{\etab\circ F} \res \p^{-1} (\cM_0)  =
\mathbf{G}_{\etab\circ f} \res \bU \cap (U_0\times \mathbb R^n).\label{e:0=0} 
\end{align} 
\end{lemma}

Before coming to the proof of the lemma, we observe that, by virtue of \cite[Theorem 5.1 \& Lemma 5.4]{DLS_Currents} it implies Theorem \ref{t:cambio} and the following ``geometric algorithm'' to find the values of $F$; see also \cite{SS17a}.

\begin{lemma}[Geometric reparametrization]\label{l:algoritmo_naturale}
The values of $F$ in Theorem~\ref{t:cambio} can be determined at any point $p\in\cM$ as follows.
Let $\varkappa$ be the orthogonal complement of $T_p \cM$. Then $p+\varkappa$ intersects $\gr (\etab\circ f)$ at a unique point $q$ and if $x := \p_{\pi_0} (q)$, then
\begin{itemize}
\item[(i)] $p \in \cM_0$ if and only if $x\in U_0$;
\item[(ii)] $p\in \cM_+$ if and only if $x\in U_+$;
\item[(iii)] $p\in \cM_-$ if and only if $x\in U_-$.
\end{itemize} 
Furthermore:
\begin{itemize}
\item[(iv)] If $p\in \cM_0$, then $F (p) = Q \a{(x, \etab\circ f (x))} = Q \, \a{q}$;
\item[(v)] If $p\in \cM_+$, then $\spt (F(p)) = \gr (f^+)\cap (p+\varkappa)$ and the multiplicity of every
point $q$ in the value $F(x)$ equals the multiplicity of the point $\p_{\pi_0}^\perp (q)$ in $f^+ (\p_{\pi_0} (q))$;
\item[(vi)] If $p\in \cM_-$, then $\spt (F(p)) = \gr (f^-)\cap (p+\varkappa)$ and the multiplicity of every
point $q$ in the value $F(x)$ equals the multiplicity of the point $\p_{\pi_0}^\perp (q)$ in $f^- (\p_{\pi_0} (q))$.
\end{itemize}
\end{lemma}

\begin{proof}[Proof of Lemma \ref{l:redistribution_lemma}] The lemma is an obvious consequence of the geometric algorithm given in \cite[Lemma 5.4]{DLS_Currents} to determine $G^+, G^-$ and $\etab\circ G$. Consider indeed a point $p\in \cM$ where the ${\rm sep} (G^+ (p)) =0$ and let $q = \etab \circ G^+ (p)$. If $\varkappa = (T_p \cM)^\perp$, \cite[Lemma 5.4]{DLS_Currents} implies immediately that 
$p+\varkappa$ intersects ${\rm Gr} (f^+)$ only in the point $q$ and that, having set $x:= \p_{\pi_0} (q)$ and $v := \p_{\pi_0}^\perp (q)$, $f^+ (x) = Q \a{v}$, so $v= \etab\circ f (x)$. This means that $p +\varkappa$ intersects the graph of $\etab\circ f$ in the point $q$, which in turn, again by \cite[Lemma 5.4]{DLS_Currents} must be precisely the value of $g$ at $p$. We have thus proved that, if ${\rm sep} (G^+ (p))=0$, then $G^+ (p) = Q \a{g (p)}$. The same argument also shows that, if ${\rm sep} (G^- (p))=0$, then $G^- (p) = Q \a{g (p)}$, thus proving condition (b) in Definition \ref{d:compatibility}. Next, we show the validity of condition (a), namely that $\min\{{\rm sep}(G^+(p)), {\rm sep}(G^-(p))\} = 0$ for every $p \in \cM$. Fix $p \in \cM$, and set again $\varkappa := (T_p\cM)^\perp$. By \cite[Lemma 5.4]{DLS_Currents}, $(p + \varkappa) \cap {\rm Gr}(\etab \circ f) = \{q\}$. If we set $x := \p_{\pi_0}(q)$, then $x \in U_+$ or $x \in U_-$ or $x \in U_0$. If $x \in U_+$, then there is $v \in \R^n$ such that $f^-(x) = Q \a{v}$, so that ${\rm sep}(f^-(x)) = 0$ and $\etab \circ f(x) = v$. Thus, $(p + \varkappa) \cap {\rm Gr}(f^-) = \{q\}$, $G^-(p) = Q \, \a{q}$, and ${\rm sep}(G^-(p)) = 0$. Analogously, one proves that if $x \in U_-$ then ${\rm sep}(G^+(p)) = 0$, and that if $x \in U_0$ then ${\rm sep}(G^+(p)) = 0 = {\rm sep}(G^-(p))$. Since $\{U_+,U_-,U_0\}$ is a partition of $U$, at each point $p \in \cM$ we necessarily have that either ${\rm sep}(G^+(p)) = 0$ or ${\rm sep}(G^-(p))= 0$, as we wanted. 

Note that, not only the argument above implies the $(G^+, G^-,g)$ satisfies the compatibility conditions of Definition \ref{d:compatibility} and hence they allow to get a well defined $F$, but they also imply immediately the conclusions (i), (ii) and (iii) of Lemma \ref{l:algoritmo_naturale}. Knowing the latter, the conclusions (iv), (v) and (vi) of Lemma \ref{l:algoritmo_naturale} are again an obvious corollary of \cite{DLS_Currents}. In turn they easily imply \eqref{e:pos=pos}, \eqref{e:neg=neg} and \eqref{e:0=0}. Finally, these three identities easily imply $\bT_F = \mathbf{G}_f \res \bU$.   
\end{proof}

\section{$L^1$ estimate on the separation over tilting planes}

We conclude with the analogue of \cite[Lemma 5.6]{DLS_Center}.

\begin{lemma}\label{l:cambio_tre_piani}
Fix $m,n,l$ and $Q$. There are geometric constants $c_0, C_0$ with the following property.
Consider two triples of planes $(\pi, \varkappa, \varpi)$ and $(\bar\pi, \bar\varkappa, \bar\varpi)$, where
\begin{itemize}
\item $\pi$ and $\bar\pi$ are $m$-dimensional;
\item $\varkappa$ and $\bar\varkappa$ are
$\bar{n}$-dimensional and orthogonal, respectively, 
to $\pi$ and $\bar\pi$;
\item $\varpi$ and $\bar\varpi$ are $l$-dimensional and orthogonal, respectively, to $\pi\times \varkappa$ and $\bar\pi\times
\bar\varkappa$.
\end{itemize}
Assume ${\rm An} := |\pi-\bar\pi| + |\varkappa-\bar\varkappa|\leq c_0$ and
let $\Psi: \pi\times \varkappa \to \varpi$,
$\bar\Psi: \bar\pi\times \bar\varkappa \to \bar\varpi$ be two maps
whose graphs coincide and such that $|\bar \Psi (0)| \leq c_0 r$ and $\|D\bar \Psi\|_{C^{0}} \leq c_0$.
Let $u: B_{8r} (0, \bar{\pi}) \to \mathscr{A}_Q (\bar\varkappa)$ be a map with
$\Lip (u) \leq c_0$ and $\|u\|_{C^0} \leq c_0 r$
and set $f (x):= \sum_i \llbracket(u_i (x), \bar\Psi (x, u_i (x)))\rrbracket$ and $\bef (x) := 
(\etab \circ u (x), \bar\Psi (x, \etab \circ u (x)))$. Then there are
\begin{itemize}
\item a map $\hat{u}: B_{4r} (0, \pi) \to \mathscr{A}_Q (\varkappa)$ such
that the map $\hat{f} (x) := \sum_i \a{(\hat{u}_i (x), \Psi (x, \hat{u}_i (x)))}$ satisfies $\mathbf{G}_{\hat{f}}
= \mathbf{G}_{f} \res \bC_{4r} (0, \pi)$
\item and a map $\hat\bef: B_{4r} (0, \pi) \to \varkappa \times \varpi$ defined by
$\mathbf{G}_{\hat\bef} = \mathbf{G}_{\bef} \res \bC_{4r} (0, \pi)$. 
\end{itemize} 
Finally, if $\beg (x) := (\etab \circ \hat{u} (x),
\Psi (x, \etab \circ \hat{u} (x)))$, then
\begin{align}
\|\hat\bef-\beg\|_{L^1} &\leq C_0 \left(\| f \|_{C^0}+ r {\rm An}\right) \big(\Dir (f) + r^m \big(\|D\bar \Psi\|^2_{C^0} + {\rm An}^2\big)\big)\, .\label{e:che_fatica}
\end{align}
\end{lemma}

\begin{proof}
We start introducing the maps $f^\pm$ and $u^\pm$. We then apply the reparametrization theorem to determine maps $g^\pm$ and $v^\pm$ which satisfy $\mathbf{G}_{g^\pm} = \mathbf{G}_{f^\pm} \res \bC_{4r} (0, \pi)$ and $\mathbf{G}_{v^\pm} =
\mathbf{G}_{u^\pm} \res \bC_{4r} (0, \pi)$. Recall that $\etab\circ f^+ = \etab\circ f^- = \etab\circ f$
and $\etab\circ u^+ = \etab\circ u^- = \etab\circ u$. By Lemma \ref{l:redistribution_lemma} and Lemma \ref{l:algoritmo_naturale}, we can next decompose $U = B_{4r}(0,\pi)$ into disjoint sets $U_+$, $U_-$ and $U_0$ by setting $U_\pm = \{x \in U \, \colon \, {\rm sep}(v^{\pm}(x)) \neq 0 \}$. Then, we define:
\begin{itemize}
\item[(a)] $\hat{u} (x) := v^+ (x)$ for $x\in U_+ \cup U_0$, so that $\hat{f} (x) = g^+ (x) = (v^+ (x), \Psi (x, v^+ (x)))$ for $x \in U_+ \cup U_0$;
\item[(b)] $\hat{u} (x) := v^- (x)$ for $x \in U_-$, so that $\hat{f} (x) = g^- (x) = (v^- (x), \Psi (x, v^- (x)))$ for $x\in U_-$;
\item[(c)] $\hat{f} (x) = f^- (x) = f^+ (x) = Q \a{\beg (x)} = Q \a{\hat\bef(x)}$  for $x\in U_0$.
\end{itemize}
Hence, if we introduce 
\begin{align}
\beg^+ &:= (\etab\circ v^+, \Psi (\cdot, \etab\circ v^+))\, ,\\
\beg^- &:= (\etab\circ v^-, \Psi (\cdot, \etab\circ v^-))\, , 
\end{align}
we easily conclude that
\begin{equation}\label{e:redistributing}
\|\hat\bef-\beg\|_{L^1 (B_{4r} (0, \pi))} =
\|\hat\bef-\beg^+\|_{L^1 (U_+)} + \|\hat\bef - \beg^-\|_{L^1 (U_-)}\, . 
\end{equation}
Now we apply \cite[Lemma 5.6]{DLS_Center} to each $f^\pm$ in order to infer
\begin{equation}\label{e:che_fatica_2}
\|\hat\bef-\beg^\pm\|_{L^1} \leq C_0 \left(\| f^\pm \|_{C^0}+ r {\rm An}\right) \big(\Dir (f^\pm) + r^m \big(\|D\bar \Psi\|^2_{C^0} + {\rm An}^2\big)\big)\, .
\end{equation}
Considering Lemma \ref{l:admissible_triples} we have $\|f^\pm\|_{C^0} \leq \|f\|_{C^0}$ and
$\Dir (f^\pm) \leq \Dir (f)$. Hence \eqref{e:che_fatica} is an obvious consequence of \eqref{e:redistributing} and \eqref{e:che_fatica_2}.
\end{proof}

\bibliographystyle{plain}
\bibliography{Biblio}

\end{document}